\documentclass[10pt]{article}
\usepackage{amsmath,amsthm,amsfonts,amssymb}

\newtheorem{theorem}{Theorem}[section]
\newtheorem{lemma}[theorem]{Lemma}

\newtheorem{prop}[theorem]{Proposition}
\newcommand{\x}{\mathbf x}
\newcommand{\y}{\mathbf y}
\renewcommand{\k}{{\rm v_{com}}} 
\newcommand{\f}{\tilde f} 
\newcommand{\wavespeed}{{\rm v}}
\newcommand{\dirset}{U}
\newcommand{\fac}{\gamma}
\newcommand{\w}{\bar\varrho}
\renewcommand{\v}{\varrho}

\newcommand{\p}{r}
\newcommand{\goo}{{\mathcal G}} 
\newcommand{\statespace}{\hat{\cal C}}

\renewcommand{\a}{\alpha}
\renewcommand{\b}{\beta}
\newcommand{\g}{\gamma}
\renewcommand{\i}{{\mathbb I}}
\newcommand{\rambda}{\ell}
\newcommand{\cu}{q}

\begin{document}

\title{Effect of Noise on Front Propagation in
Reaction-Diffusion equations of KPP type}
\author{Carl Mueller$^1$
\\Department of Mathematics\\University of Rochester
\\Rochester, NY  14627
\\E-mail:  cmlr@math.rochester.edu
\and
Leonid Mytnik$^2$
\\Faculty of Industrial Engineering and Management
\\Technion -- Israel Institute of Technology
\\Haifa 32000, Israel
\and
Jeremy Quastel$^3$
\\ Departments of Mathematics and Statistics
\\University of Toronto
\\Toronto, Ontario
\\Canada
}
\maketitle

\footnotetext[1]{Supported by an NSF grant. 
$\mbox{}^2$ Supported in part by the Israel Science Foundation 
(grant No. 1162/06).
$\mbox{}^3$ Supported
by NSERC, Canada.  


AMS 1991 {\em subject classifications}
Primary, 60H15; Secondary, 35R60, 35K05.}




\begin{abstract}
We consider reaction-diffusion equations of KPP type in one spatial dimension,
perturbed by a Fisher-Wright white noise, under the assumption of uniqueness
in distribution.  Examples include the randomly perturbed Fisher-KPP equations
\begin{equation}
\partial_t u = \partial_x^2 u + u(1-u) + \epsilon \sqrt{u(1-u)}\dot W,
\end{equation}
and
\begin{equation}
\partial_t u = \partial_x^2 u + u(1-u) + \epsilon \sqrt{u}\dot W,
\end{equation}
where $\dot W= \dot W(t,x)$ is a space-time white noise.

We prove the Brunet-Derrida conjecture that the speed of traveling fronts
is asymptotically \begin{equation}
2-\pi^2 |\log \epsilon^2|^{-2}
\end{equation}
 up to a factor of order $ (\log|\log\epsilon|)|\log\epsilon|^{-3}$.
\end{abstract}

\section{Randomly perturbed KPP and the Brunet-Derrida conjecture}
\label{intro}
\setcounter{equation}{0}

In this article we study randomly perturbed Kolmogorov-Petrovsky-Piscunov (KPP) equations,
\begin{equation}
\label{rkpp}
\partial_t u 
= \partial_x^2 u + f(u) + \epsilon\sigma(u)\dot{W} \qquad t\ge 0, ~~x\in \mathbf R, ~~ u\ge 0
\end{equation}
where $\dot{W}=\dot{W}(t,x)$ is two-parameter white noise; $f$ is assumed to be  a Lipschitz function satisfying standard KPP  conditions,
\begin{equation} \label{standardkppcond}
f(0)=f(1)=0; \qquad 0< f(u)\le uf'(0), \quad u\in (0,1),
\end{equation} and in addition that for $u\ge 1$,
\begin{equation}\label{Gammadef}
f(u) \le 2-u.
\end{equation} We can and will rescale so that $f'(0)=1$. We assume that 
$\sigma^2(u)$ is a Lipschitz function satisfying 
\begin{equation}\label{sigmacond0}
\sigma^2(u)\le u
\end{equation}
 and for which there exist $a^*>0$ and $0<u^*<1$ such that 
\begin{equation}\label{sigmacond}
{\sigma^2(u_2)-\sigma^2(u_1)}\ge a^*({u_2-u_1}), \qquad{\rm for } \quad 0\le u_1<u_2\le u^*.
\end{equation}
We will consider (\ref{rkpp}) with initial data 
 $u(0,x) =u_0(x)$ satisfying, for some $x_0\in {\mathbf R}$, 
 \begin{equation}\label{16}
 u_0(x)\ge \theta>0, ~x\le x_0,\qquad{\rm and}\qquad \int_{x_0}^\infty u_0(x)dx<\infty
\end{equation}
and contained in some subset $\statespace$ of the set $\mathcal{C}_{\rm exp}$  of non-negative continuous functions $f$ on $\mathbf{R}$
with $f(x)\le Ce^{|x|}$ for some $C<\infty$, 
for which we know 
\begin{eqnarray}\label{wu1}
&&  u_0(x)\in \statespace ~~~\Rightarrow ~~~u(t,x)\in \statespace \;\;\forall t>0,\;P-{\rm a.s.}\\
  \label{wu2}&&  {\rm Weak~ uniqueness~ holds~ in~}\statespace.
  \end{eqnarray}
 Key examples are \begin{equation}\label{MScase1}f(u)=u(1-u)\end{equation}and \begin{equation}\label{MScase2}\sigma^2(u)= u(1-u){\bf 1}(u\le 1)\end{equation} or 
\begin{equation}\label{con}\sigma^2(u) = u\end{equation} 
with $u_0(x)\in [0,1]$ for (\ref{MScase2}) and $u_0$ satisfying 
$
 e^{-x}\ge u_0(x)$
for (\ref{con}).
(\ref{MScase1}) with (\ref{MScase2}) appears
 as the limit of the long range voter model and with 
 (\ref{con})
 as the limit of the long range contact process (see 
\cite{mt95}).   

Note that (\ref{Gammadef}) is not relevant for models such as (\ref{MScase2}) where $0\le u(t,x)\le 1$ for all time.  But some condition
on large $u$ is needed in cases such as (\ref{con}) where fluctuations
can take the solution above $1$.

We regard the stochastic partial differential equation (SPDE) (\ref{rkpp}) as shorthand for the  integral equation,  
\begin{eqnarray}
\label{int-eq:u}
u(t,x) &=& \int G(0,y,t,x)u_0(y)dy
    + \int \int_{0}^{t} G(s,y,t,x)f(u(s,y))ds dy \nonumber  \\
&&    + \int \int_{0}^{t} G(s,y,t,x)\sigma(u(s,y)) W(ds, dy),  
\end{eqnarray}
where $G(s,y,t,x) = G(t-s,y-x)$ is the heat kernel
\begin{equation}
\label{e17}
G(t,x) =  (4\pi t)^{-1/2}\exp\{-{x^2}/{4t}\},
\end{equation} and the white noise $\dot W$ is defined by specifying, for square integrable deterministic 
functions $\varphi(s, y)$, that
$W(\varphi):=\int \int_0^\infty \varphi(s, y)W(ds ,dy)$
are a Gaussian family with mean zero and covariance 
\begin{equation}
E\left[ W(\varphi)W(\psi)\right] =
\int \int_0^\infty  \varphi(s, y)\psi(s, y) ds dy .
\end{equation}
Here, and throughout, $\int f dx$ means the integral over the
entire real line $\mathbf R$.
Solutions to (\ref{int-eq:u}) are called {\it mild solutions}.  See \cite{wal86} for the definition of the stochastic integral in (\ref{int-eq:u}).  Readers unfamiliar with SPDE
can think of the following system of ordinary stochastic differential equations on ${\mathbf R}^{\frac{1}{n}{\mathbf Z}}$,
\begin{equation}\label{lattice}
du_{i/n} = [n^{2} ( u_{(i+1)/n}-2u_{i/n}+u_{(i-1)/n}) + f(u_{i/n})] dt + 
n^{1/2}\epsilon \sigma( u_{i/n}) dB_{i/n},
\end{equation}
where $B_{i/n}(t)$ are independent standard Brownian motions, and $n$ is large.
A corresponding evolution of functions on $\mathbf R$ is produced by
connecting the points $(i/n, u_{i/n})$ and $((i+1)/n, u_{(i+1)/n})$ by straight lines, and (\ref{rkpp}) is obtained in the weak limit as $n\to \infty$.

When  $\epsilon=0$, (\ref{rkpp}) with $f$ of the form (\ref{MScase1}) is the standard KPP, or Fisher-KPP equation, introduced in 1937 by both
Fisher \cite{F}, 
 and  Komogorov, Petrovskii, and Piscuinov \cite{KPP}.  The basic facts in this case are:    There is a one-parameter family $F_{\wavespeed}$ of traveling front solutions 
$F_\wavespeed (x-\wavespeed t)$ with $F_\wavespeed$ decreasing, 
$F_\wavespeed (x) \to 1$ as $x\to -\infty$,  $F_\wavespeed (x) \to 0$ as 
$x\to \infty$ and $F_\wavespeed(x) \simeq e^{-\gamma x}$ for large $x$, with 
$\wavespeed = \gamma+\gamma^{-1}$.  For initial data $u_0(x)={\bf 1}(x\le 0)$  we 
have convergence to the traveling front with minimal speed,
\begin{equation}\label{wvspd}
\wavespeed_0=2
\end{equation} in the sense that
\begin{equation}\label{emstar}
\lim_{t\to\infty}\sup_{x\in\mathbf{R}}
|u(t,x)-F_{\wavespeed_0} (x-m^*(t))| = 0\end{equation} where $m^*(t)$ defined by 
$u(t,m^*(t))=1/2$ satisfies
$\lim_{t\to\infty} m^*(t)/t = \wavespeed_0$. 
Further details about convergence of the KPP solution to the 
traveling front were given by McKean \cite{mck75} and \cite{mck76}, and Bramson 
\cite{bra78} and \cite{bra83}, 
among many others.   

When $\epsilon>0$ with initial data in $\mathcal{C}_{\rm exp}$ satisfying (\ref{16}), one has non-negative, continuous solutions, with a finite upper bound on the support
\begin{equation}\label{ellarr}
r(t) = \sup\{x\in {\mathbf R} :   u(t,x) >0\},
\end{equation}
for $t>0$.
The process viewed from $r(t)$,
\begin{equation}\label{pvff}
\tilde u(t,x) = u(t,x+r(t))
\end{equation}
should have a unique nondegenerate stationary solution.  This is the
{\it random traveling front}.  One also expects $t^{-1}r(t)$ to have  a
nonrandom limit,
\begin{equation}
\label{e4}
\wavespeed_\epsilon = \lim_{t\to\infty}t^{-1} r(t).
\end{equation}
This was proved \cite{ms95} in the case (\ref{MScase1}),  (\ref{MScase2}), for sufficiently small $\epsilon$.  They consider
 initial 
data $0\le u_0(x)\le 1$  such that $u_0(x)=1$ 
for $x<\ell$ and $u_0=0$ for $x>r$ for some $-\infty<\ell\le r<\infty$. 
Because in this case $\sigma(1)=0$, solutions stays within this class, and the result
of \cite{ms95} extends to any $f$ and $\sigma$ satisfying in addition to our assumptions, that
$\sigma(u)=1 +\mathcal O(\sqrt{1-u})$ as $u\uparrow 1$.

We now make some comments to justify the form (\ref{sigmacond0}),(\ref{sigmacond})  of the stochastic perturbation.  The most important reason for taking \begin{equation}\sigma(u)\simeq \sqrt{u},\qquad 0\le u <\!< 1\end{equation} 
is that this is the type of correction seen when approximating the reaction-diffusion
problems by microscopic particle models.  
For example, the reaction-diffusion equation with $f(u)=u(1-u)$ was originally derived by Fisher
as a model for the spread of an advantageous gene;  the term $u(1-u)$ in (\ref{MScase1}) represents the frequency of 
mating between the individuals with and without the advantageous gene.
If there is randomness in the mating, for example, if matings were successful with a certain probability, the variance of the random
term is naturally  proportional to $u(1-u)$, and this leads to a term $\sqrt{u(1-u)}\dot{W}(t,x)$.   
%

In this article we are primarily concerned with the asymptotics of $\wavespeed_\epsilon$ as $\epsilon\to 0$ in (\ref{rkpp}).
It is not hard to see (for example, in (\ref{MScase1}), by taking expectations, and applying Jensen's inequality) that $\wavespeed_\epsilon
\le \wavespeed_0$.
Recently, Brunet and Derrida \cite{bd97} and \cite{bd01} 
(see also \cite{kns98}, \cite{pl99}) have made the remarkable conjecture that as $\epsilon\to0$,
\begin{equation}
\label{bd:conjecture}\wavespeed_0-
 \wavespeed_\epsilon \simeq  \frac{\pi^2}{|\log\epsilon^2|^{2}}.
\end{equation}
It is worth noting how enormous the correction is.  For example, a
naive Taylor expansion might suggest, since symmetry implies $\epsilon=0$ is a local maximum,
that $\wavespeed_0-
 \wavespeed_\epsilon \simeq \mathcal O(\epsilon^2)$.  
 The phenomenon was unexpected, and first observed
 through computer simulations of particle systems.  It was  not long before this was understood at the physical level as a consequence of the
pulled nature of the fronts. Recall that in an evolution equation with traveling fronts
between an  unstable and a stable state, the front is said to be {\it pulled} if its asymptotic
speed is the same as that of the linearization of the equation about the unstable state,
and {\it pushed} if the speed is larger than that of the linearization (see \cite{vS}).
KPP equations have (marginally) pulled fronts.   Because in pulled fronts the front speed is determined in the
  region where the density $u$ is very small, in retrospect one should not be surprised that  fluctuations there of order
  $\sqrt{u}$, would have a dramatic effect on the front speed.
Bramson \cite{bra78} also proved for the KPP equation with initial data $u_0(x) = {\bf 1}(x\le 0)$, that, 
\begin{equation}\label{bramson}
m^*(t) = \wavespeed_0 t + \frac32 \log t + \mathcal O(1),
\end{equation}  which is also supposed to be universal for pulled fronts \cite{vS}.  These behaviors all reflect the fact that in pulled fronts there is no spectral gap in the linearization around $F_{\wavespeed_0}$.  

The phenomenon (\ref{bd:conjecture}) has  also been observed in systems where the variable $u$ is forced to take discrete values,  such
as particle systems on the lattice with random walks and birth-death components.  Here
$\epsilon^2$ is the effective mass of a particle. 

Brunet and Derrida \cite{bd97} conjectured that the front speeds in these systems  behave for small $\epsilon$
like that of the solution of the cutoff KPP equation
\begin{equation}\label{cutoffKPP}
\partial_t u 
= \partial_x^2 u + u(1-u) {\mathbf 1}(u\ge \epsilon^2).
\end{equation}
The idea is that when $u<\epsilon^2$, $u(1-u)< \epsilon\sqrt{u(1-u)}$ and the 
noise term in (\ref{rkpp}) beats the creation term down to zero.  Alternatively (\ref{cutoffKPP}) can
be thought of as a single particle cutoff.  They
then gave a nonrigorous argument that for small $\epsilon$, (\ref{cutoffKPP}) 
has travelling fronts with velocity
\begin{equation}\label{cut}
\wavespeed_{\rm cutoff} \simeq  \wavespeed_0-\frac{\pi^2}{|\log\epsilon^2|^{2}}.
\end{equation}
  
  The argument for (\ref{cut}),
using matched asymptotics, is not difficult to make rigorous.  It is known 
\cite{DPK}, \cite{BDL} that \begin{equation}\label{cut2}
\wavespeed_{\rm cutoff}- \wavespeed_0+\frac{\pi^2}{|\log\epsilon^2|^{2}}=\mathcal O( \frac1{|\log\epsilon|^{3}}).
\end{equation}
Implicit in our argument is a simpler proof (with a slightly worse correction of $\mathcal O( \frac{\log|\log\epsilon|}{|\log\epsilon|^{3}})$).  

What was less clear was how to make rigorous the connection between either microscopic particle models or (\ref{rkpp}) and (\ref{cutoffKPP}).  Here we work with (\ref{rkpp}) as
a kind of canonical model system for the phenomenon
(\ref{bd:conjecture}):  In particular, the fact
that particle models and (\ref{rkpp}) both are expected to display this behaviour is perhaps the
strongest motivation for the particular form of
the noise (\ref{sigmacond}). 

See \cite{P} and references therein for a very comprehensive review of the 
physical aspects of the Brunet-Derrida theory.

Conlon and Doering \cite{cd04} recently obtained progress on 
(\ref{bd:conjecture}) by coupling (\ref{MScase1}), (\ref{MScase2}) to a contact process
(see Liggett \cite{lig85}), proving that for sufficiently small $\epsilon$, 
\begin{equation}
\label{e9}
\wavespeed_\epsilon \ge \wavespeed_0-
 \frac{C\log|\log\epsilon|}{|\log\epsilon|^{2}}.
\end{equation}

Very recently, \cite{BDMM} have made a conjecture about the corrections to (\ref{bd:conjecture}).  Using a phenomenological
argument, they propose
\begin{equation}\label{bdmm}
\wavespeed_\epsilon - \wavespeed_{0} +\frac{\pi^2}{|\log\epsilon^2|^{2}}\simeq \frac{6\pi^2\log|\log
\epsilon|}{|\log\epsilon^2|^3}.
\end{equation}
The $6$ on the right hand side is closely related to the $3$ in (\ref{bramson}).

In this article we prove
the Brunet-Derrida conjecture for models of the form (\ref{rkpp})-(\ref{Gammadef})  with a correction of the same order as conjectured
in   \cite{BDMM}. 

However our understanding of  the well-posedness of (\ref{rkpp})
is not complete, and so a few comments are needed before we can
state the result.

Existence for (\ref{rkpp}) is straightforward and can be obtained as the limit of
its spatial discretization (\ref{lattice}).   
Starting with non-negative initial data, one obtains in this way a non-negative solution, H\"older $\alpha<1/2$ in space and $\beta<1/4$ in time.
Alternately, equations of the form (\ref{rkpp}) can be obtained as appropriate limits of particle systems.   Note that we are allowing solutions to have $u\ge 1$, which is slightly non-standard, in particular
for models such as $f(u)=\sigma^2(u)=u(1-u)$ where $u$ is usually
taken to be in $[0,1]$.  In terms of existence, this does not make any
difference.

On the other hand, uniqueness of (\ref{rkpp}) is not known in our case because the coefficient in front of the noise in not Lipschitz.  At the time of writing there is not even a consensus whether strong uniqueness should be true (for new results on strong uniqueness for stochastic partial differential equations with non-Lipschitz coefficients 
 see~\cite{MP}, although they still do not cover the present case).  {\it Weak uniqueness} means uniqueness of the martingale
problem for (\ref{rkpp}) with respect to the family of functionals
$f_\phi(u)= \exp\{ -\int u\phi dx \}$, $\phi$ smooth, non-negative, with compact 
support, within the class of continuous non-negative solutions $u$,
and in addition, the measurability of the Markov transition functions
$P_{s,u(\cdot)}(A)= P(u(t,\cdot)\in A~|~u(s,\cdot)=u(\cdot))$.
Weak uniqueness in particular implies the 
strong Markov property, which is one of our basic tools.  
 Weak 
 uniqueness  can be obtained in special cases of (\ref{rkpp}) using duality.
An example is when $\sigma$ is of the form (\ref{MScase2})
or (\ref{con}).  The case (\ref{MScase1}), (\ref{MScase2}) has an explicit dual particle system, described below,
whose existence allows one in principle to compute the law for the stochastic
partial differential equation.
The case (\ref{MScase1}), (\ref{con}) is self-dual (see \cite{HT}). 

What we really use about our solutions are the strong Markov
property with respect to a family of hitting times, together with the {\it comparison principle}. Roughly the comparison principle for SPDE states that if $u$ and $v$ are  solutions of
\begin{equation}\label{spde5}
\partial_t u =\partial_x^2 u +f(x,u) + \sigma(u) \dot W, \qquad \partial_t v =\partial_x^2 v +g(x,v) + \sigma(v) \dot W
\end{equation}
on $[0,T]$ with $v(0,x)\le u(0,x)$, and $g(x,u)\le f(x,u)$ for all $x,u$,
 then $v(t,x)\le u(t,x)$ for all $t\in [0,T]$ and $x\in \mathbf R$ almost surely. It is the analogue of monotonicity or attractiveness in particle systems.
  Examples of such theorems can be found in \cite{Pa}. A simple variant of the above it that our $v$ will, in addition, satisfy a Dirichlet condition $v=0$ on a set $\dirset$ such as $x>{\rm v} t$ or
 $|x|\ge {\rm v} t+L$.  This is not a large leap, as one can think of it as the $N\to \infty$ limit of $g=-N$ on $\dirset$. So there is no surprise that the comparison continues to hold.  There will be a few other twists and we will
 be a little more precise later.  But the main point is that proofs of comparison theorems of this type 
 require as input a strong uniqueness theorem.  Hence they are not directly available to us.

 Now any solution we are really interested in will be the result of some approximation scheme by systems, for example particle systems, for 
 which the comparison principle is essentially obvious.  Similarly, the
 strong Markov property will hold for such systems.  And both are maintained  under weak limits.  So we could in principle just take a pragmatic approach and
 simply assume that our solution has the needed properties. Since this is a little cumbersome, instead we 
 will state our results under the assumption of weak uniqueness.
 
 Note that  weak uniqueness implies the existence of versions satisfying both the strong Markov property and comparison principle.   That it implies the strong Markov property is
 well known.  To obtain a version satisfying the comparison principle, construct a sequence of
 Lipschitz $\sigma^{(n)}(u)$ converging uniformly to $\sigma(u)$.  The
 corresponding equations have strong uniqueness and therefore the 
 comparison principle.  It is not hard to check that such sequences are
 tight and it is easy to see that the comparison principle continues to 
 hold in the weak limits.   Note that all our results are statements in distribution. It is therefore
 always enough to work with appropriate versions of our process, and therefore weak uniqueness is
 sufficient.


 We can now state the main theorem.  Let $u(t,x)$ be the solution of (\ref{rkpp}) and $r(t)$ be as in (\ref{ellarr}).  For initial data in $\mathcal{C}_{\rm exp}$ satisfying (\ref{16})  let
\begin{equation}\label{130a}
\bar{\rm v}_\epsilon = \limsup_{t\to \infty} t^{-1} r(t).
\end{equation} 
For initial data satisfying $u(0,x)\ge \theta>0$, $x\le 0$, $u(0,x)=0$, $x>0$, let
\begin{equation}\underline {\rm v}_\epsilon=
\liminf_{t\to \infty} t^{-1} r(t).
\end{equation} Let
$\alpha(a) $ be the largest $\alpha$ such that 
\begin{equation}\label{alpha}
(1-a)u1(u\le \alpha)\le f(u).
\end{equation}
Note that $\alpha(a)>0$ if $a>0$ from the assumptions on $f$.
For example, if $f''(0)>-\infty$ we have $\alpha(a) = Ca$ for some $C<\infty$.

\begin{theorem}
\label{t1}  Assume  $u_0(x)\in \statespace$ satisfying (\ref{16}), (\ref{wu1}), (\ref{wu2}).   Then there exists $\epsilon_0>0$ such that for all
$\epsilon\leq \epsilon_0$, \begin{eqnarray}\label{e10}
   \underline{\rm v}_\epsilon & \ge & \wavespeed_0 -
\frac{\pi^2}{|\log\epsilon^2|^{2}} -\frac{2\pi^2[
 9\log|\log\epsilon| -\log \alpha(|\log\epsilon|^{-3}) }{ |\log\epsilon^2|^{3}}   \\\label{e11}  \bar{\rm v}_\epsilon
 & \le &  \wavespeed_0 -
\frac{\pi^2}{|\log\epsilon^2|^{2}}  + \frac{8\pi^2\log|\log\epsilon|}{ |\log\epsilon^2|^{3} }.
\end{eqnarray}
\end{theorem}

{\it Remark on $\epsilon_0$.}   Since the phenomena is observed in 
particle simulations, it is worthwhile to ask whether the mathematical result can be proven 
with $\epsilon_0$ of a size approachable by computation.  In fact, computations with effectively $N=10^{10}$ particles are commonplace at the time of writing.
 In case (\ref{MScase1}), (\ref{MScase2}),  we 
can check that our method works with $\epsilon_0= e^{-11}$.    Since $N$ particles corresponds to a correction of size $\epsilon=N^{-1/2}$ the mathematical result covers the typical regime of computations.




{\it Remark on  duality for the case of (\ref{rkpp}) with $f(u)=\sigma(u)= u(1-u)$. }    Let $x_i(t): i=1,\ldots,N(t)\le \infty$ be a system of independent Brownian motions with generators $\partial_x^2$.  
Each particle splits in two at rate $1$, and pairs of particles coalesce at 
exponential rate $\epsilon^2$ during their  intersection local time.  
The generator is
\begin{equation}
A f(\mathbf{ x}) = \sum_{i} \Delta_i f(\mathbf{ x})
+\sum_{i} ( f( {\mathbf x}_i^+ ) -f (\mathbf{x})) + {\epsilon^2} \sum_{i> j} 
{\delta}_{x_i =  x_j}( f({\mathbf x}_i^- ) -f({\mathbf x}))
\end{equation}
where ${\mathbf x}_i^+$ is the configuration obtained from $\mathbf x$ by replacing $x_i$ by two particles at the same location,
and  ${\mathbf x}_i^-$ is the configuration obtained from $\mathbf x$ by removing $x_i$.
We have the  duality relation \cite{shi88},
\begin{equation}
\label{e41} 
E\left[\prod_{i}(1-u(t,x_i(0)))\right]
= E\left[\prod_{i}(1-u(0,x_i(t)))\right], 
\end{equation}
where the expectation is taken over independent $u$ and $x_i$.  Among other things, (\ref{e41}) gives us an expression for 
the moments of $u$, providing the weak uniqueness. 

One can also deduce from the result about the wavespeed in the random KPP  
results  about the wavespeed in the dual process.   Suppose we start our branching and coalescing system with one particle at $0$, and let
$L(t)$, $R(t)$  denote the positions of the leftmost and rightmost  particles in the system at time $t$. Take $u_0(x) = {\mathbf 1}(x\le 0)$.
The duality relation, together with the natural reflection symmetry and
spatial homogeneity, give 
$P(L(t) <-x)=P(R(t)>x) =  E[ u(t,x) ] $, and Theorem \ref{t1} then translates to \begin{equation}\label{branch-coalesc}
\lim_{t\to \infty} t^{-1}R(t) = -\lim_{t\to \infty} t^{-1}L(t)=\wavespeed_\epsilon.
\end{equation}

Here then is another example of the Brunet-Derrida theory:  The branching-coalescing Brownian motions model possesses two invariant measures.  The stable one is a Poisson point process with intensity $\epsilon^{-2}$, and the unstable one consists of no
particles at all.   On a large scale we see the first invading the second at linear 
speed $\wavespeed_\epsilon$.  If we introduce a phase variable in $[0,1]$ so that
$0$ corresponds  to the unstable phase and $1$ corresponds to the stable phase, then
the effective particle mass is $\epsilon^2$, as predicted.



%


Finally we comment on the structure of the paper.  To make the arguments leading to Theorem \ref{t1} more
transparent, in the next  section
we sketch the logic of the proof, assuming the main technical lemmas, which are then
left for later sections, and assuming as well that the necessary manipulations of the
SPDEs can be performed.  We then prove 
the validity of these manipulations in Section \ref{comparison theorem}.

\section{Outline of the proof}
\setcounter{equation}{0}

\subsection{Comparison equation}
\renewcommand{\nu}{\varepsilon^2}

The general idea behind our proof of Theorem \ref{t1} is to compare the stochastic KPP evolution (\ref{rkpp}) to:
\begin{equation}\label{compeq}
\left\{\begin{array}{ll}
\partial_t \varrho
= \partial_x^2 \varrho + f(\varrho), & x< {\rm v} t, \\
\varrho(t,x) = 0,& x\ge{\rm v} t. \end{array}\right.
\end{equation}
We search for the ${\rm v} = \k$ for which there exists a traveling front solution 
\begin{equation}
\varrho(x,t) = F(x-\k t)
\end{equation}
 with $\lim_{x\to -\infty} F(x)=1$ and \begin{equation}\label{bccomp}
 F'(0)= -\nu.
 \end{equation}

The problem (\ref{compeq})-(\ref{bccomp}) is our replacement for Brunet and Derrida's comparison equation (\ref{cutoffKPP}).  The idea is that the solution will have a mass
of $\mathcal O(\varepsilon^2)$ within a distance $\mathcal O(1)$ of $x=\k t$.  Here $\varepsilon$ is a small parameter which does not necessarily have to be related to the perturbation parameter $\epsilon$ in (\ref{rkpp}).  In fact, it will be convenient to take $\varepsilon$ to be
slightly larger or smaller than $\epsilon$.  But
if $\varepsilon=\epsilon$, the mass $\mathcal O(\varepsilon^2)$  is the critical mass which can survive in the stochastic equation when $u$ is small.  Heuristically,  this  will provide a consistent strategy for a stochastic traveling front in (\ref{rkpp}) to propagate.

To determine the resulting ${\rm v}=\k(\nu)$, 
let $\x(t)= F(-t)$ and note that the problem is equivalent to that of finding the
 ${\rm v}$ such that  the solution of the ordinary differential equation
 \begin{equation}
\x''={\rm v}\x' -{ f}(\x),\qquad \x(0)=0, ~\x'(0)= \nu
\end{equation}
 has $\x(\infty)=1$.
 Let $\y= \x'$.  In the  phase plane of
 \begin{eqnarray} \x'& = &\y \\ \y' & = & {\rm v} \y -{ f}(\x)
 \nonumber 
 \end{eqnarray} there is an unstable node at
 $(\x,\y)=(0,0)$ and a saddle point at $(\x,\y)=(1,0)$, joined by a 
  separatrix solution $(\x(t),\y(t))$, $-\infty<t<\infty$, with $\x(-\infty)=0$, $\x(\infty)=1$.
 For ${\rm v}\ge 2$, the linearization $\x'_{\rm lin}= \y_{\rm lin}$, $\y'_{\rm lin}= {\rm v} \y_{\rm lin} -\x_{\rm lin}$
 about $(0,0)$ has distinct positive eigenvalues which merge as ${\rm v}\downarrow 2$, then split into a complex pair for ${\rm v}<2$.  For ${\rm v}\ge 2$, the separatrix 
 corresponds to an exponentially decaying traveling front solutions of the (nonrandom) KPP equation.  For ${\rm v}<2$,
 $(0,0)$ is a spiral source, the $\x\ge 0$ corresponding in the same way to a traveling front
 solutions of (\ref{compeq}).  The separatrix enters the region $(\x,\y)\in 
 (0,1)\times (0,\infty)$ at $(\x,\y)= (0,\nu({\rm v}))$, and problem 
 (\ref{compeq})-(\ref{bccomp}) is now seen to be equivalent to computing
the inverse function ${\rm v}(\nu)$.


  \begin{prop}\label{kbdprop} Let $\k=\k(\nu)$ be the solution of (\ref{compeq})-(\ref{bccomp}). There exist $\varepsilon_0>0$ such that for $|\varepsilon|<\varepsilon_0$,
\begin{equation}\label{kappabd}
 2- \frac{\pi^2}{(|\log\nu|-3\log|\log\nu| -\log\alpha(|\log\nu|^{-3}) -2)^{2}} 
\le \k \le
 2- 
\frac{\pi^2}{\left( |\log \nu| + 3\right)^{2}}.
\end{equation}
If $\alpha(a)=a$ then $\varepsilon_0$ can be taken to be $e^{-8}$.
\end{prop}

 First we give the heuristic idea of the proof.  It is not hard to see that ${\rm v}(\nu)$ is monotone decreasing in $\varepsilon$ with  ${\rm v}\downarrow 2$ as $\varepsilon^2\downarrow 0$.     The linearization about $(0,0)$ has explicit solution
 \begin{equation}\label{xlin}
 \x_{\rm lin} (t)= \nu \hat\delta^{-1/2} \exp\{\left(1-{\scriptstyle\frac{\delta}{2}}\right) t \}\sin(\hat\delta^{1/2} t)
 \end{equation}
 where \begin{equation}
 \delta = 2-{\rm v}.
 \end{equation} and 
$\hat \delta =\delta-\frac{\delta^2}4$.  
   Let $\Theta$ be the smallest $t>0$ with
 \begin{equation}\label{Theta}
 \x_{\rm lin}(\Theta)=1.
 \end{equation}
 Note that $\Theta\sim |\log\nu|$.
 Pretending the linearization is meaningful globally,
 one would want  \begin{equation}\label{thetabd}\x'_{\rm lin}(\Theta)=0.\end{equation}
  (\ref{Theta}) and (\ref{thetabd}) become, with $C_1=1$, $C_2= (1-{\scriptstyle\frac{\delta}{2}})$,
  \begin{eqnarray} \label{kbdequation}
\nu \hat\delta^{-1/2} \exp\{\left(1-{\scriptstyle\frac{\delta}{2}}\right) \Theta \}\sin(\hat\delta^{1/2} \Theta)=C_1\\\nonumber  \nu \exp\{ \left(1-{\scriptstyle\frac{\delta}{2}}\right) \Theta \}(-\cos(\hat\delta^{1/2} \Theta)) = C_2.
 \end{eqnarray}
 (\ref{kbdequation}) gives a nonlinear equation for $\delta$ in terms of $\epsilon$ from which it
 is simple to obtain estimates like
  (\ref{kappabd}).
 The only difference in the
 rigorous proof is that we will use sub- and super-solutions to get (\ref{kbdequation}), but  with 
 slightly worse $C_1$ and $C_2$.     
 
\begin{proof}[Proof of Proposition \ref{kbdprop}] First of all note that ${\rm v} = \k$ depends monotonically on $f$:  If $f_1(u)\le f_2(u)$ for all $u$ then
 the corresponding ${\rm v}(f_1)\le {\rm v}(f_2)$. 
 
{\it Upper bound.} Consider (\ref{compeq})-(\ref{bccomp})  with $f$ replaced by 
\begin{equation}\label{barf}
\bar f(u) = \left\{\begin{array}{ll}
u &\mbox{if $u \in [0,1]$} \\
2-u &\mbox{if $u>1$}.  
\end{array}
\right.
\end{equation}
The corresponding ${\rm v}=\k$ is larger  as  $\bar f\ge f$.  Call $\delta=2-{\rm v}$ and assume momentarily that $\delta< 0.2$.
The solution to 
\begin{equation}
\x''={\rm v} \x' -{\bar f}(\x),
\end{equation}
can be computed explicitly.  For
$t\le\Theta=\inf\{ t>0:\x(\Theta)=1\}$, 
$
\x\left(t\right) = \x_{\rm lin}(t)$ from (\ref{xlin}).  Now we 
consider the phase plane $(\x,\x')$.
One checks that the linearization around the saddle $(\x,\x')=(2,0)$ has stable direction  $({\varpi},1)$ and unstable direction $({\lambda}, 1)$
where, 
\begin{equation}\label{214} {\varpi}(\delta)= 1 -\frac{\delta}2 - ( 2 - \hat\delta )^{1/2},\qquad
{\lambda}= 1-\frac{\delta}2 + ( 2-\hat\delta)^{1/2} .
\end{equation}

A separatrix solution $(\x,\x')$ joins $(0,\nu)$ to $(2,0)$. It must
coincide with the stable line $
y={\varpi}x-2{\varpi}$ in the region $\x>1$, because the equation 
is linear there.  So in order for $(\x(\Theta), \x'(\Theta))$
to lie on the separatrix  we must have \begin{equation}
 \x'_{\rm lin}(\Theta)= {\varpi}\x_{\rm lin}(\Theta)  -2{\varpi},\end{equation} which is equivalent to  (\ref{kbdequation}), with
$C_1= 1$ and $C_2= -{\varpi}$.

 {\it Lower bound.}  
 Let
$\alpha  =\alpha(|\log\nu|^{-3})$ from (\ref{alpha}) and $a=1-2|\log\nu|^{-3}$ and define
\begin{equation}\label{sharpf}
 \f(u) = \left\{\begin{array}{ll}
a u &\mbox{if $u\le \alpha/2$} \\
a(\frac{\alpha}{2}- u)&\mbox{if $u>\alpha/2$}.  
\end{array}
\right.
\end{equation}
Consider the problem (\ref{compeq})-(\ref{bccomp}) with $f$ replaced by
$ \f$.
The corresponding $ {\rm v}={\rm v}(\f)$ is smaller than $\k$  because  $\f\le f$,
  Call $\delta=2- {\rm v}$.  From the upper bound we know that for sufficiently small $\nu$,
$
\hat{\delta}=\delta-\frac{\delta^2}4>0
$, in which case
the solution of \begin{equation}\label{lbode}
 \x''-  {\rm v} \x' +\f(\x)=0
 \end{equation}
  with $\x(0)=0$ and $\x'(0) = \nu $ is $\x(t)=\x_{\rm lin}(t)$ from
  (\ref{xlin})
for 
$t\le  \Theta$ when $\x_{\rm lin}(\Theta)=\alpha/2$. 
%
%
%
In order to lie on the separatrix joining the unstable fixed point $(\x,\x')=(0,\nu)$ to the saddle point $(\alpha ,0)$, we must have $(\x(\Theta), \x'(\Theta))$  on the stable line 
$
y=(1-\sqrt{2-a}) (x- \alpha) $.  This gives (\ref{kbdequation}) with $C_1= \alpha/2$ and $C_2=(\sqrt{2-a}-1-\frac{\delta}{2}) \alpha/2$. 

{\it Proof  of (\ref{kappabd}).}   Assume $\varepsilon_0$ is sufficiently small that  $\delta<1$ and.  Dividing the two equations in  (\ref{kbdequation})
 gives $-\tan(\hat\delta^{1/2}\Theta) = \hat\delta^{1/2}(C_1/C_2)$. Let  $\Theta=\hat\delta^{-1/2} \pi - \beta$.  Since $\frac{4}{\sqrt{2}\pi} x
 \le \tan x\le \sqrt{2} x$ if $x\le \pi/4$ the solution has $0\le C_1/(\sqrt2 C_2) \le \beta \le  \pi/4$.  Now the second equation of (\ref{kbdequation}) gives,
 \begin{equation}\label{preprop}
 \hat\delta \left(1-{\scriptstyle\frac{\delta}{2}}\right)^{-2} =
 \pi^2 ( |\log\nu| + \log C_2  +|\log\cos(  \hat\delta^{1/2}\beta)| +\beta)^{-2}
 \end{equation}
To get a lower bound on $\k=2-\delta$, drop the non-negative terms $|\log\cos(  \hat\delta^{1/2}\beta)|$ and $\beta$ from the right hand side and
note that $\delta\le\hat\delta \left(1-{\scriptstyle\frac{\delta}{2}}\right)^{-2}$ and
$\log C_2 = \log \{(\sqrt{ 2- a}-1)\alpha/2\}\ge -|\log\alpha| - 3\log|\log\nu| - 2\log 2$.   To get an
upper bound, note first that if $\varepsilon_0$ is sufficiently small, then
from the lower bound we have just described,
$ \hat\delta\le 1$.  Since
$\beta\le  \pi/4$, we then have
$ |\log\cos(  \hat\delta^{1/2}\beta)| \le 1$.  Also $\log C_2= \log( \sqrt{2-\hat\delta}-1-\delta/2) <0$.
Finally, if we take $\varepsilon_0$ sufficiently small, then the lower bound we just proved gives
$\delta< 10|\log\nu|^{-2}$ so that  $\hat\delta (1-\frac{\delta}{2})^{-2}\ge \delta (1+  10|\log\nu|^{-2})^{-1}$
and then (\ref{preprop}) gives
\begin{equation}
\delta \ge \pi^2( |\log\nu| + 2)^{-2} -100 |\log\nu |^{-4}
\end{equation}
which gives (\ref{kappabd}) for $\varepsilon_0$ small enough.
\end{proof}

 \subsection{Upper bound}

Consider $u$ satisfying (\ref{rkpp}) with initial data 
\begin{equation}
u(0,x)=\bar F(x)
\end{equation}
where $\bar\varrho(t,x) =\bar F (x-{\rm v}t)$ is a traveling front solution of \begin{equation}\label{compeq3}
\left\{\begin{array}{ll}
\partial_t \bar\varrho
= \partial_x^2 \bar\varrho + \bar f(\bar\varrho), & x< {\rm v} t, \\
\bar\varrho(t,x) = 0,& x\ge{\rm v} t, \end{array}\right.
\end{equation}
with $\bar f$ as in (\ref{barf}),
 \begin{equation}\label{twonineteen}
{\rm v} = \k(\nu) + |\log\epsilon|^{-3} 
\end{equation}
and \begin{equation}\label{epar}
-\bar F'(0)=
\nu = \fac\epsilon^2\end{equation}
with $0<\gamma <\!\!< 1$ to be chosen.
 
 $\bar F$ is a modified version of a traveling front from the comparison problem (\ref{compeq})-(\ref{bccomp}), with 
 a larger $\bar f>f$, a slightly larger speed, and lying slightly above the 
 separatrix connecting $(0,0)$ to $(2,0)$.   There is some convenience in using $\bar f$ instead of $f$.  It is convex.  Also, some things are explicitly
 computable.  For example,
 \begin{equation}
\label{F}
\bar F(x)= 
\left\{\begin{array}{ll}
0, \qquad           x\ge 0 ; &\\
\x_{\rm lin} (-x) ,\qquad  -\Theta\le x<0 ; &   \\
\kappa{\lambda} e^{-{\lambda}(\Theta+x)} + 2-(1+\kappa)
e^{-{\varpi}(\Theta+x)}, ~~ x< -\Theta,   & \end{array}
\right.
\end{equation} 
where $\kappa =\kappa(\nu)$ is chosen such that $F$ and $F'$ are continuous at 
$x=-\Theta$ and $\x_{\rm lin}$ and $\Theta$ are defined in (\ref{xlin})-(\ref{thetabd}).  Keep in mind that these all depend on $\epsilon$ though the dependence is not written explicitly.  One can check that $\kappa\simeq (1-2^{-1/2})|\log\nu|^{ -2}$. 
Note also that the modification of the speed in (\ref{twonineteen}) is smaller 
than the ${\mathcal O}(\frac{\log|\log\epsilon|}{|\log\epsilon|^3})$ error 
terms in the main result, Theorem \ref{t1}.
 




%
Fix  a positive integer
 $T$ and an $L>0$ and consider the hitting time
\begin{equation}\label{xidef}
\xi = \inf\{ t\in [0,T] : u(t,x) \ge \bar F(x-{\rm v} t-L){\rm ~ for ~some~ }  x\in\mathbf R\}
\end{equation}
We run $u$ up to time $\xi$, and then restart with new initial data
\begin{equation}
\bar F(x-{\rm v}\xi - L),
\end{equation} a shift of $L$ from the original comparison front. By the 
strong Markov property and the comparison theorem (Proposition \ref{comparison}), we obtain an upper bound 
$\bar u$ on the original solution of (\ref{rkpp}) with initial data $\bar{F}$.  
Repeating the process, we inductively define  a sequence of stopping times 
$\xi_{i+1}\in [\xi_i, \xi_i+T]$, and 
an upper bound $\bar u$ for all time on the solution $u$  with initial data 
$\bar F (x)$.  $\bar u$  satisfies (\ref{rkpp}) on $(\xi_i,\xi_{i+1})$
with initial data $u(\xi_i,x) = \bar F(x-{\rm v}\xi_i - Li)$. 
 
Suppose we can show that 
\begin{equation}\label{prexi}
P(\xi<T) <1/2.
\end{equation}
By the law of large numbers, the speedup of the front of $\bar u$ over that of $u$ is by a factor ${L}/{E[\xi]}$.  But from (\ref{prexi}), $E[\xi] \ge T/2$.
We obtain in this way, using (\ref{kappabd}), (\ref{twonineteen}) an upper bound on $\bar{\rm v}_\epsilon$ defined in (\ref{130a}), 
\begin{eqnarray}\bar {\rm v}_\epsilon & \le & {\rm v} +2T^{-1}L\nonumber\\ \label{preupper}
 &\le &
2- 
{\pi^2} \left( \log \varepsilon^{-2} +3\right)^{-2} + |\log\epsilon|^{-3}+  2T^{-1}L.
\end{eqnarray}
If we choose 
\begin{equation}\label{tog}
T^{-1}L \le |\log\epsilon|^{-3}
\end{equation}
we obtain the upper bound (\ref{e11}) for
  initial data  bounded above by $\bar F$.

There is  a tradeoff between $T$ and $L$.  Large $L$ in principle makes (\ref{prexi}) easier, because $\bar{F}(x-{\rm v} t -L)$ is increasing in $L$.  But then (\ref{tog}) forces us to choose $T$ large,
and it becomes harder to control $u$ on the long time interval to obtain (\ref{prexi}). 



For more general initial data, satisfying only $\int_0^\infty u_0(x)dx<\infty$, we can use the fact that at any time $t>0$, $r(t)<\infty$ a.s. 
This is proved in the special case $\sigma^2(u)=u$ in  \cite{Mueller-Perkins},
but it is well-known that the method can be adapted without too
much work to cover the present situation.  We do not give details here.  This means that we can bound $u(t,\cdot)$ by a shift
of $\bar F$, and obtain the upper bound (\ref{e11}) as before.  


This reduces the upper bound to (\ref{prexi}).  The main idea to prove 
(\ref{prexi}) is to split the  solution $u(t,x)$ of (\ref{rkpp}) 
with initial data $\bar F$ into 
\begin{equation}
u=v+w
\end{equation}
where $v(t,x)$ is the mass which does not cross $x={\rm v} t$;
\begin{equation}
\label{drkpp}\left\{\begin{array}{ll}
\partial_t v 
= \partial_x^2 v +  f(v) + \epsilon\sigma(v)\dot{W_1}, & x<{\rm v} t, \\
v(t,x) = 0,& x\ge {\rm v} t, \end{array}\right.
\end{equation}  with $v(0,x) = \bar F(x)$, 
and $w\ge 0$ is the rest.  $\dot{W_1}$ will be another space-time white noise. As usual, the SPDE is interpreted in the mild sense;
\begin{eqnarray}
\label{e26}
v(t,x) &=& \int G_{\rm v}(0,y, t,x){\bar F}(y)dy         
+ \int \int_{0}^{t}G_{\rm v}(s,y,t,x) f(v(s,y))ds dy 
               \nonumber\\
&& + \epsilon\int \int_{0}^{t}G_{\rm v}(s,y, t,x)
                 \sigma(v(s,y))W_1(dsdy)  
\end{eqnarray}
 where $G_{\rm v}(s,y,t,x)$ is the sub-probability density for a Brownian motion
with generator $\partial_x^2$, starting at $y$ at time $s$, to end at $x$ at time
$t>s$ never having entered the region $\{z\ge {\rm v} u\}$ for times $s\le u\le t$.    In Proposition \ref{comparison} in Section \ref{comparison theorem} it is shown that we can find a probability space on which such a  splitting holds.  

 We expect the solution $v(t,x)$ of
(\ref{drkpp}) to remain close to the solution $\varrho(t,x)$ of
the deterministic comparison equation (\ref{compeq}) with the same initial
data $\bar F(x)$.  Because it is a subsolution of (\ref{compeq3}) we have $\varrho(t,x)\le \bar F(x-{\rm v}t)$.
If $L$ is large enough, we can therefore expect 
$v$ not to hit $\bar F(x-{\rm v}t-L)$ for some time.

The key point now is that  if $F'(0)<\!<\epsilon^2$ then $w$ is so negligible that $u=v+w$ does not hit $\bar F(x-{\rm v} t-L)$ for some time either.  
To prove this,  we need a 
better way to represent $w$. 
One can also view the Dirichlet boundary condition in (\ref{drkpp}) as a removal of mass.
Let $A(t)$ be the mass which is removed at the boundary $x={\rm v} s$ in (\ref{drkpp}) during the time interval $0\le s\le t$.  Then we have another representation for $v$ satisfying (\ref{drkpp}) or (\ref{e26}):
\begin{equation}
\label{equ:v1}
\partial_t v= \partial_x^2 v  + f(v) 
       + \epsilon\sigma(v) \dot{W}_1 - \delta_{x-{\rm v}t}\dot A.
\end{equation}

We would like to write an equation for $w$, with a new white noise
$\dot W_2$, independent of $\dot W_1$.  If $\dot W_1$ and $\dot W_2$ are independent white noises, then \begin{equation}\sigma_1\dot W_1+\sigma_2 \dot W_2= \sqrt{\sigma^2_1+\sigma^2_2}~\dot W\end{equation} where $\dot W$ is a white noise. Hence the equation for $w$ should read,
\begin{equation}
\label{e30first}
\partial_t w = \partial_x^2 w  +   f(v+w) - f(v)
      + \epsilon\tilde\sigma\dot{W}_2+ \delta_{x-{\rm v} t}\dot A,
\end{equation} 
with initial data $w(0,x)\equiv 0$, where
\begin{equation}
\tilde\sigma(t,x, w)= \sqrt{ \sigma^2(v(t,x)+w)-\sigma^2(v(t,x)) }.
\end{equation} 
But this is only reasonable as long as $\sigma^2(v(t,x)+w)-\sigma^2(v(t,x)) $ 
remains non-negative.
In Proposition \ref{comparison} of Section \ref{comparison theorem}, it is 
shown that there exists a  probability space on which there  are white noises 
$W_1$ and $W_2$ for which (\ref{e30first}) holds, up to a stopping time
\begin{equation}
\tau= \inf\{ t\ge 0 : \exists x,~ \tilde\sigma(t,x,w(t,x))=0, ~ w(t,x)\neq 0\}, 
\end{equation}
after which the desired noise coefficient $\tilde\sigma$ might cease to make sense.

By the comparison theorem, and since $ f(v+w) - f(v)\le \|f\|_{\rm Lip} w$,  up to time $\tau$ we have $u-v=w
\le \bar w$ almost surely,
where
\begin{equation}
\label{e30a}
\partial_t \bar w = \partial_x^2 \bar w  + \|f\|_{\rm Lip}\bar w
      + \epsilon\tilde\sigma(\bar w)\dot{W}_2+ \delta_{\{x={\rm v}t\}}\dot A.
\end{equation} 
    As long as \begin{equation}
    \label{lbonfluc}\tilde\sigma^2(\bar w)\ge a^* \bar w,
    \end{equation} this is basically  a superprocess with an injection of mass at $\{ x={\rm v} t\}$.  The critical input of mass in such an equation is easily calculated to be $\mathcal O(\epsilon^2)$.
    In other words, if the rate of mass entering is $ o(\epsilon^2)$, then it is being killed by the noise in time
    $\mathcal O(1)$ with very high probability. And it suffices to show just that the 
{\em   expected} incoming mass $E[ A(t+1)-A(t)]$ is $ o(\epsilon^2)$. 

  To get such a bound,
note that by comparison $v\le \bar v$, the solution of
\begin{equation}
\label{drkpp2a}
\partial_t \bar v
= \partial_x^2 \bar v + \bar f(\bar v) + \epsilon\sigma(\bar v)\dot{W}, \qquad x \le {\rm v} t
\end{equation}
with $\bar v=0 $ for $x\ge {\rm v} t$.
Take expectation in (\ref{drkpp2a}) and use the  concavity of $\bar f$ to see that  $E[v]$ is a subsolution of (\ref{compeq3}).
 In particular,
\begin{equation}\label{vlew}
E[v(t,x)]\le \bar F(x-{\rm v} t).
\end{equation}
This can be translated into a bound
on the expected rate of incoming mass $A(t)$ as follows. Taking expectation in (\ref{equ:v1}), \begin{eqnarray} 
\label{pto}
E\left[ A(t+1)- A(t)\right] & = & \int q(t,y, t,t+1)E[v_t(y)]dy \\ \nonumber 
&&  + \int\int_t^{t+1} q(s,y,s,t+1)E[{f}(v(s,y))]dsdy
\end{eqnarray}
 where $q(s,y, u,t) = P_{s,y}(\exists r\in (u,t]: B_r \ge {\rm v}r)$ for a Brownian motion $B_r$ with 
 generator $\partial_x^2$.
Using $E[f(v)]\le E[\bar f(v)]\le \bar f(E[v]) \le \bar f(\bar \rho)$, and that $v_0(y)=\bar F(y)$, we see that 
\begin{eqnarray} 
E\left[ A(t+1)- A(t)\right] & \le & \int q(t,y,t,t+1)\bar F(y-{\rm v}t)dy \\ \nonumber 
&&  +
  \int\int_t^{t+1} q(s,y,s,t+1)\bar f(\bar F(y-{\rm v}s)) dsdy.
\end{eqnarray}
Now $\bar F(x-{\rm v} t)$ is a traveling front 
solution of (\ref{compeq}) with $\bar f$ instead of $f$.  The rate of mass removal at the boundary is proportional to the slope $\nu$ at the boundary, and hence there
is a $ C_{(\ref{bdnearkt})}<\infty$ such that,
\begin{equation}
E\left[ A(t+1)- A(t)\right]\leq  
 C_{(\ref{bdnearkt})}\nu . \label{bdnearkt}
\end{equation}

The only difficulty is maintaining (\ref{lbonfluc}).  By (\ref{sigmacond}) it 
holds as  long as $v+\bar{w}\le u^*$.  What we will do is obtain an a priori 
estimate that $v\le u^*/2$ in a strip ${\rm v}t-M \le x\le {\rm v} t$.  This 
is reasonable since we know that $v$ is close to $\rho$, which is of 
$\mathcal O(\varepsilon^2)$ there.   If \begin{equation}\label{eeeem}
M=M(u^*)\end{equation} is chosen sufficiently large, 
we can then iteratively show that $\bar{w}\le u^*/2$, and furthermore that it does 
not support the complement of a strip ${\rm v}t-1 \le x\le {\rm v} t+1$ around 
our proposed front.  This then provides us with sufficient noise to show that 
$w$ is negligible there as well. 

%

%

%
To fix  $\gamma$ and $T$, let us explain very briefly the iterative
procedure. Take $T$ to be an integer and divide up the time interval $[0,T]$ into intervals of length $1$.
The mass arriving in $[n,n+1)$ and evolving
according to (\ref{e30a}), is bounded by (\ref{bdnearkt}).
It dies before time $n+2$ with probability at least $
1-c_0\gamma$ where $c_0=c_0(a^*,\|f\|_{\rm Lip}) 
$. So
this happens for every $n=0,1,2,\ldots,T-1$, with
probability at least
\begin{equation}
\label{236}
1- c_0\gamma T.
\end{equation}
In order to have the probability in (\ref{236}) greater than $\frac{3}{4}$, 
we thus take 
\begin{equation}
\label{epar2}
\gamma =\frac14 c_0^{-1}T^{-1}.
\end{equation} 
To fix all our constants  we note that if 
\begin{equation}\label{elldef}
L=|\log\epsilon| + \log|\log\epsilon|
\end{equation} 
then we have
\begin{equation}
F(x-L)\ge F(x) + 3{\lambda} e^{-{\lambda}x}.
\end{equation}
By (\ref{tog}) we need
\begin{equation}\label{teeee}
T =\lfloor |\log\epsilon|^4 \rfloor.
\end{equation}
We have explained how the  upper bound is a consequence of the
following lemmas.

\begin{lemma} 
\label{lem1} 
Let $v$ be the solution of (\ref{drkpp}) with initial
data (\ref{F}).  Let $\varrho$ be the solution of (\ref{compeq3}) with the
same initial data.  Let $\gamma,L,T,M$ be as in (\ref{eeeem})-(\ref{teeee}).  Then
\begin{equation}\label{prechi}
P\left(\exists t\in [0,T] :   v(t,x)>\bar\varrho(t, x)+   3 {\lambda}
e^{-{\lambda}(x-{\rm v} t)}{\rm ~for~some~}x\in\mathbb R
\right)\le 1/16.
\end{equation}
\end{lemma}

\begin{lemma} \label{lemma2} Under the same conditions as in Lemma
\ref{lem1},
\begin{equation}
\label{253}
P\left(\exists t\in [0,T] :   v(t,x)>u^*/2 {\rm ~for~some~}x\in({\rm v} t -M,{\rm v} t)
\right)\le 1/8.
\end{equation}
\end{lemma}

\begin{lemma} \label{lemma3a} Under the same conditions as in Lemma \ref{lem1}, let
$\bar w$ be the solution of (\ref{e30a}) where $A$ is defined in (\ref{equ:v1})
\begin{equation}
\label{254aa}
P\big(\exists t\in [0,T] :  
\sup_{x\in [{\rm v} t-1,{\rm v} t+1]^c}  \bar w(t,x)>0{\rm ~~or}\sup_{x\in [{\rm v} t-1,{\rm v} t+1]} \bar w(t,x)\geq u^*/2\big)  \le 1/4.
\end{equation}
\end{lemma}

\subsection{Lower bound}\label{proof of the lower bound}

The proof of the lower bound uses a more standard method; comparison to 
oriented percolation  \cite{bradurr88}.  Similar arguments were used by Conlon 
and Doering \cite{cd04} to prove their lower bound.  The improvement here comes 
from the use of the special comparison front from (\ref{lbode}) and refined 
large deviation estimates.

  For ${\rm v},L>0$ let
$
G_{{\rm v},L} (s,y,t,x)
$
 be the sub-probability density in $x$ at time $t$ for a Brownian motion $B_u$, $s\le u\le t$
with generator $\partial_x^2$, starting at $B_s=y$,  killed if it  enters the region $|z|\ge  {\rm v} u +L$, $s\le u\le t$. If $\varrho(t,x) $ is a given function, let
\begin{equation}
(\goo \varrho)( t,x) =\int \int_0^t G^2_{{\rm v}, L} (s,y,t,x)
\varrho(s,y) dsdy 
\end{equation} Note that this makes sense since we are in one dimension.  The lower bound is based on the following simple lemma about the deterministic equation.
Let 
\begin{equation}\label{underlinef}
 \underline{f}(u) = \left\{\begin{array}{ll}
(1-|\log\epsilon|^{-3})u &\mbox{if $u\le \alpha/2$} \\
(1-|\log\epsilon|^{-3})(\frac{\alpha}{2}- u)&\mbox{if $u>\alpha/2$}.  
\end{array}
\right.
\end{equation}
From the definition (\ref{alpha}) of $\alpha$, we have  \begin{equation}\underline{f}(u) \le f(u) \qquad {\rm whenever} \quad f(u)\ge 0.\end{equation} Let 
\begin{equation}
\varepsilon^2 = \gamma \epsilon^2
\end{equation}
with $1<\!\!<\gamma$ and
 ${\rm v}={\rm v}_{\rm com}(\varepsilon)$ as in (\ref{compeq}) - (\ref{bccomp}).

\begin{lemma} 
\label{lemma3}   There exist $\epsilon_0>0$,  $C_{(\ref{2.12})}<\infty$, $0<L\le |\log\epsilon|$, and  
${\underline \v}_0( x)$ supported on $[-L,L]$ with $0\le{\underline \v}_0(x)\le \alpha(|\log\epsilon^2|^{-3})$ such that if 
  $\epsilon<\epsilon_0$, $\gamma \ge C_{(\ref{2.12})}|\log\epsilon|^{10} $ and $r>1$,
 the solution ${\underline \v}(t,x)$ 
on $0\le t\le 1$ of 
\begin{equation}
\label{comparison:eq:lb}
\left\{\begin{array}{ll}
\partial_t {\underline \v} = \partial_x^2 {\underline \v} + {\underline f}({\underline \v}) , &
           |x|<L+{\rm v} t  \\
 {\underline \v}(t,x) = 0,&  |x|\ge L+{\rm v} t,      \end{array}
\right.
\end{equation}   
with ${\underline \v}(0,x) ={\underline \v}_0( x)$, and $\underline{f}$ as in 
(\ref{underlinef}), satisfies, for
${\rm v}'={\rm v} +  |\log\epsilon|^{-3} $,
    \begin{equation}
{\underline \v}(1,x) - r\epsilon \sqrt{ {\goo }
\underline\v(1,x) } \ge {\underline \v}_0( x- {\rm v}' )
\qquad x\in [{\rm v}'-L,{\rm v}'+L].\label{2.12}
\end{equation}
\end{lemma}

\begin{proof}  
%
%
%
We follow the notation and construction from the proof of the lower bound of Proposition \ref{kbdprop}.  Let  $\alpha$, $a$, $\x$ and $\Theta$ all be as in the proof of the lower bound of Proposition \ref{kbdprop}.   
We claim that (\ref{2.12}) holds with  $L=\Theta$ and
\begin{equation}
{\underline \v}_0(x) = \left\{\begin{array}{ll}
\x(L-|x|) &
         0\le |x|<L, \\
0  & |x|\ge L  .       \end{array}
\right. 
\end{equation}  
To prove this, note that the solution $\underline{\v} (t,x)$ of 
(\ref{comparison:eq:lb}) 
satisfies  
$\underline{\v} (t,x) \le \bar{q}(t,x) = e^{ t|\log\epsilon|^{-3}} \min\{ \x(L+{\rm v} t - x), \x(L)\}$ since ${\underline \v}_0(x)\le \x(L-x)$,  $x\in {\bf R}$,
and $\bar{q}(t,x)$ is a supersolution of (\ref{comparison:eq:lb}), 
and $\underline{\v} (t,x) \ge \underline{q}(t,x)$ where \begin{equation}
\underline{q}(t,x) = \left\{\begin{array}{ll}
e^{ t|\log\epsilon|^{-3}}\x(L+{\rm v} t-|x|) &
           {\rm v} t \le |x|<L+{\rm v} t  \\
e^{ t|\log\epsilon|^{-3}}\x(L) &  |x|\le{\rm v} t \\
0  & |x|\ge {\rm v} t +L        \end{array}
\right. 
\end{equation} 
since $\underline{q}(t,x)$ is a subsolution of (\ref{comparison:eq:lb})  with the same initial data. 
So
\begin{eqnarray} && {\underline \v}(1,x) -{\underline \v}_0(x-{\rm v}') \ge 
\underline{q}(1,x)- {\underline \v}_0(x-{\rm v}') \\&&~
\ge   \left\{\begin{array}{ll}
\x_{\rm lin}(L+{\rm v}-|x|)-\x_{\rm lin}(L+{\rm v}'-|x|) &
         |x|\in ({\rm v}, {\rm v}'+L], \\
(e^{|\log\epsilon|^{-3}} -1)\x_{\rm lin}(L)  & |x|\le {\rm v}  ,       \end{array}
\right. \nonumber
\end{eqnarray}
Here we used that $\x(t) = \x_{\rm lin}(t) $ for $t\le L$.    For   $|x|\in ({\rm v}, {\rm v}'+L]$ we now use that 
if $\delta<1/2$ then $\x'_{\rm lin} \ge  \varepsilon^2e^{ \frac12 (L+{\rm v} - |x|)}$  there
and
 ${\rm v}-{\rm v}' = |\log\epsilon|^{-3}$ to get a lower bound.
 For $|x|\ge {\rm v}$ we use $e^{|\log\epsilon|^{-3}} -1\ge |\log\epsilon|^{-3}$ and   $\x_{\rm lin}(L) = \alpha/2$.
This gives
 \begin{equation} 
 {\underline \v}(1,x) -{\underline \v}_0(x-{\rm v}')\ge \left\{\begin{array}{ll}
|\log\epsilon|^{-3}\varepsilon^2 e^{\frac12 (1-\frac{\delta}{2})(L+{\rm v} - |x|)} &
         |x|\in ({\rm v}, {\rm v}'+L], \\
|\log\epsilon|^{-3}\alpha/2  & |x|\le {\rm v}  ,       \end{array}
\right. 
 \end{equation}
From the explicit form of $\x_{\rm lin}$, and using $G_{{\rm v},L}\le G$, as long as  $\exp\{|\log\epsilon|^{-3}\}\le 2$ and $\delta<1$,
\begin{eqnarray}
(\goo \bar{q})( 1,x) & \le  & 2 \int_0^1 \int_{-\infty}^{{\rm v} s} \frac{ e^{ -\frac{ (y-x)^2}{2(1-s)} }}{4\pi(1-s)} \min\{ \varepsilon^2 \delta^{-1/2} e^{{\rm v} s - y+L} , \alpha/2\} dsdy \nonumber\\
& \le & \min\{4 \varepsilon^2 \delta^{-1/2} e^{-(x-L-{\rm v})},\alpha\} .
\end{eqnarray}
Hence one can check that $ {\underline \v}(1,x) -{\underline \v}_0(x-{\rm v}')\ge r\epsilon \sqrt{\goo \bar{q}}$ as long as 
\begin{equation}\label{cond17}
r\le \min\{ 6\gamma^{1/2}|\log\epsilon|^{-7/2},  \epsilon^{-1}|\log\epsilon|^{-3} \sqrt{\alpha( |\log\epsilon|^{-3})/2} \}.
\end{equation}
 Since $\goo \bar{q} \ge \goo \underline\v$ we are done.
\end{proof}

From the comparison theorem (Proposition \ref{comparison}), we can construct a probability space on which the solution 
of 
\begin{equation}
\label{rkpplower}
\left\{\begin{array}{ll}
\partial_t \underline{u} 
= \partial_x^2 \underline{u}  + {\underline f}(\underline{u} ) + \epsilon\sigma(\underline{u} )\dot{W} , &
           |x|<L +{\rm v} t\\
 \underline{u} (t,x) = 0,&  |x|\ge L+{\rm v} t.       \end{array}
\right.
\end{equation} gives an almost sure lower bound  for the solution 
of  (\ref{rkpp}) with the same initial data. 

Suppose we start (\ref{rkpplower}) with ${\underline \v}_0( x)$ from Lemma \ref{lemma3}.
The idea is that the solution $\underline{u} $ will stay close to ${\underline \v}$ up to time $1$. 
To see how close, let us make a very rough argument.  Since $\underline{f}$ has Lipschitz constant
$1$ one expects that for times of $\mathcal O(1)$,
$
|\underline{u}(t,x)-\underline{\v}(t,x)|$ is controlled by something like $\epsilon|Z_{L,{\rm v}}(t,x)|$ where\begin{equation}\label{azee}Z_{L,{\rm v}}(t,x)=
\int \int_0^t e^{t-s} G_{{\rm v},L} (s,y,t,x)
\sigma(u(s,y)) W(dsdy).
\end{equation}
The actual bound is somewhat more complicated, but it amounts to the same thing.
Recall that $\sigma^2(u)\le u$.  By concavity of $\underline f$, $E[u] \le \underline \v$.  So
\begin{equation}
E[Z_{L,{\rm v}}^2(1,x)]
\le e{\goo }
\underline\v(1,x). \end{equation}
Things are tight in the region close to the front $\{ x={\rm v} t+L\}$ where $\underline\v=\mathcal O(\varepsilon^2)$.  
Here ${\goo }
\underline\v(1,x)=\mathcal O(\varepsilon^2)$ as well.  Hence the fluctuations of $\underline{u}-\underline{\v}$ are of order $\epsilon\varepsilon=\gamma\epsilon^2$ there.  This is the reasoning behing (\ref{2.12}). The key point of the following refined large deviation estimate is that it shows that the fluctuations near the front are of $\mathcal O(\gamma\epsilon^2)$ instead of the $\mathcal O(\epsilon)$ one would obtain naively.

\begin{lemma} \label{lbd}  Let $\underline{\v}$, $L$ be as in Lemma \ref{lemma3}.  There exists $C_{(\ref{268})}<\infty$ such that for all $0< \p<\epsilon^{-1}$,
\begin{equation}\label{268}
P\Big( \exists x: \underline{u}(1,x) \le \underline\v(1,x) - \p\epsilon \sqrt{{\goo }
\underline\v(1,x)} \Big)\le 4L\exp\{ - C^{-1}_{(\ref{268})} \p^2\} . 
\end{equation}
\end{lemma}
Note the  factor $L$ on the right hand side is because the
large deviations are done on space intervals of size $1$, and
then summed over the width of $\underline\v(1,x)$.  If we take \begin{equation}\label{cond18}\p^2 = C_{(\ref{268})}[|\log\epsilon|^3+2\log|\log\epsilon|+\log4],
\end{equation}
 the right hand
side is less than $|\log\epsilon|^{-1}e^{-|\log\epsilon|^{3}}$.
%
We conclude that if we start (\ref{rkpplower}) 
with  ${\underline \v}_0(x) $,
then
\begin{equation}
P( u(1,x) \ge{\underline \v}_0( x- {\rm v}  ),~~ x\in \mathbf  R ) \ge 1-|\log\epsilon|^{-1}e^{-|\log\epsilon|^{3}}.
\end{equation}
Now we can ask for this to happen $T=|\log\epsilon|$ times, to obtain
\begin{equation}
P( u(T,x) \ge {\underline \v}_0( x- {\rm v}   T), ~x\in \mathbf R )
\ge 1- e^{-|\log\epsilon|^{3}}\stackrel{\rm def}{=} p.
\end{equation}
By symmetry we also have
\begin{equation}
P( u(T,x) \ge {\underline \v}_0( x+ {\rm v}  T),~ x\in \mathbf R )
\ge p.
\end{equation}

This allows us to compare the system to a 2-dependent oriented percolation.  
Let $\mathcal L= \{(m,n)\in \mathbf Z^2: m+n$ is even, $m\ge 0\}$.  Let
$\hat {\mathcal L}$ denote the set of directed bonds $(m,n)\to (m+1,n+1)$ or $(m,n)\to(m+1, n-1)$.  Let $X_b$, $b\in \hat{\mathcal L}$ be random variables
taking values in $\{0,1\}$. Assume that for all $b\in   \hat{\mathcal L}$,
\begin{equation}
P( X_b=1)\ge p.
\end{equation}
Also assume that 
 $X_{b}$ and $X_{b'}$ are independent if the lattice distance between  $b$ to $b'$ is strictly larger than $2$.
 If $m_1<m_2$ then we say $(m_1,n_1)\to (m_2,n_2)$ if they are joined by a
 sequence of directed bonds $b_i\in\hat {\mathcal L}$ with $X_{b_i}=1$.
Let $\mathcal S$
denote the  subset of $(m,n)\in\mathcal L$ such that $(0,x)\to (m,n)$ for some $x<0$.

\begin{lemma} \label{oplem} Suppose that 
\begin{equation}\label{pie}
p> 1-e^{-10}.
\end{equation}
Then, with
probability $1$,  for all but finitely many $m$, \begin{equation}\label{badbd}
N_m \stackrel{\rm def}{=} \max\{n: (m,n)\in \mathcal S \}\ge (1- {10}{|\log(1-p)|^{-1}}) m
\end{equation} 
\end{lemma}

\begin{proof} Using the standard contour counting argument one obtains
\begin{eqnarray} \label{cdbd} &
P( N_m < m (1-\delta)) \le \sum_{n\ge m\delta} 4^{2n+m} (1-p)^{n/2}.&
\end{eqnarray} 
(see \cite{cd04}, Lemma 3.6, for a complete proof of (\ref{cdbd}).  The only difference is the exponent $n/2$ on
the $(1-p)$ which comes from the 2-dependence.)  Take $\delta = {10}{|\log(1-p)|^{-1}}$.  If we assume (\ref{pie}) then it is not hard to
check that the right hand side is less than $2e^{-M}$ and the result follows from Borel-Cantelli.
\end{proof}

{\it Remark.}  (\ref{badbd}) is very far from optimal:  It is known  (see \cite{D}, \cite{GP}) that for $p$ close to
$1$, the speed of oriented percolation, 
$\lim_{m\to \infty} \frac{N_m}{m}= 1- {\mathcal O}(1-p)$.  If one uses the stronger result, one can
check for the case (\ref{MScase1}), (\ref{MScase2}), the main result holds with $\epsilon_0= e^{-11}$.

\medskip

{\it  Proof of the lower bound.}  Start (\ref{rkpp}) with initial data $u_0$ 
in $\mathcal{C}_{\rm exp}$ satisfying (\ref{16}). Without loss of generality, $x_0=0$.  As long as $\theta>\alpha$ we have
$u_0(x)\ge \sum_{n=-\infty}^0 {\underline\v}_0(x-nT)$.  We say 
$X_{(m,n)\to(m+1, n\pm 1)}=1$
if $u(mT,x+nT), u((m+1)T, x+(n\pm 1)T)\ge {\underline\v}_0(x)$. 
Recall $r(t)=\sup\{ x\in \mathbf R: u(t,x) >0\}$.  If $N_m\ge am$,
then $r(mT)\ge amT$, and furthermore, 
$r(t)\ge (am-1)t - L$, $(m-1)T\le t\le mT$.  Hence 
from Lemma \ref{oplem} we get a lower bound 
\begin{equation}
\underline{\rm v}_\epsilon\ge {\rm v}(\varepsilon^2) - 10|\log(1-p)|^{-1}= {\rm v}(\varepsilon^2) - 10|\log\epsilon|^{-3}.
\end{equation}

\section{Comparison}\label{comparison theorem}
\setcounter{equation}{0}

We now state precisely the comparison theorem we are using. Let $U$ denote  the set
\begin{equation}
\label{Gupper}
\dirset= \{ (t,x) \in [0,T]\times \mathbf R~:~ x\le {\rm v} t \}
\end{equation}
 for some ${\rm v}>0$ \begin{prop}  \label{comparison}  Assume   (\ref{wu1}), (\ref{wu2}). 

1. Suppose that $g(t, u)\le f(u)$ are Lipschitz functions, $u\in \mathbf R $, $t\ge 0$ 
   and initial data
 $v_0(x)\le u_0(x)$, $x\in\mathbf R$ are given.
 There exists a probability space $(\Omega,\mathcal F, P)$ on which there are
 white noises $\dot{W}, \dot{W_1}$, 
a solution  $u$  to  (\ref{rkpp}); a solution  $v$  to\begin{equation}
\label{drkpp2}\left\{\begin{array}{ll}
\partial_t v 
= \partial_x^2 v + g(v) + \epsilon\sigma(v)\dot{W_1}, & (t,x)\in \dirset, \\
v(t,x) = 0,& (t,x)\not\in \dirset, \end{array}\right.
\end{equation}  with $v(0,x) = v_0(x)$, and satisfying
 \begin{equation}
u(t,x)\ge v(t,x),\qquad x\in \mathbf R, \quad t>0.
\end{equation}
2.  Fix possibly random ${\cal F}_0$-measurable $u_0\in  \statespace$.
 Suppose that $g(u)$ is  the Lipschitz function and there is also another 
    initial data $\bar u_0=v_0\in \mathcal{C}_{\rm exp}$ such that 
\[ u_0(\cdot)\leq \bar u_0(\cdot),\;\;{\rm a.s.,}\]
and $ \bar u_0(x)=0$ for all $x\geq 0.$
 Then there exists a probability space $(\Omega,\mathcal F, P)$ on which white noises 
 $\dot{W}, \dot{\bar W}, \dot{W_1}, \dot{W_2}$ and a vector of processes 
 $(u,\bar u, v, \bar w)$ are defined and satisfy the following properties: 

(i)
 $u$  is a solution  to  (\ref{rkpp}); 

(ii)
$\bar u$  is a solution  to  (\ref{rkpp}) in ${\cal C}_{\rm exp}$ starting at $\bar u_0$ with $\bar W$ replacing $W$; 

(iii)
$v$  is a solution  to  (\ref{drkpp2}) in ${\cal C}_{\rm exp}$ starting at $v_0=\bar u_0$  where $U$ is of the form 
 (\ref{Gupper});

(iv)  $\dot W_2$ is a space-time white noise independent of $\dot W_1$;

(v)
Up to time \begin{equation}
\tau= \inf\{ t\ge 0 : \exists x,~ \sigma_w^2(t,x)=:\sigma^2(\bar u)-\sigma^2(v)=0, ~ \bar u(t,x)\neq v(t,x)\}, 
\end{equation}\begin{equation}
\bar u-v\le \bar w \qquad {\rm and} \qquad 
 u\le \bar u\qquad {\rm a.s.},
\end{equation}
where
\begin{equation}
\label{e30}
\partial_t \bar w = \partial_x^2 \bar w  +\bar w
      + \epsilon\sigma_w\dot{W}_2+ \delta_{x-\k t}\dot A,
\end{equation} 
and
\begin{equation}\nonumber
A(t)  = -\int (v(t,x)-v(0,x)) dx  +\int\!\int_0^t g(v(s,x)) ds dx
+\epsilon\int\!\int_0^t \sigma(v(s,x)) W_1(ds dx).
\end{equation}


\end{prop}

 \begin{proof} 1.  Assume first that  $\sigma$  is  globally Lipschitz.
Then the proof goes essentially along the lines of Theorem 3.1 of \cite{Mueller-Perkins}. One approximates the solutions by lattice versions as in (\ref{lattice}), for which the ordering is elementary.  Then one shows the
ordering is
preserved in the limit.  Because one has strong uniqueness it means the
solution of the SPDE's are ordered in the desired way.
Now suppose we do not have the strong uniqueness. We construct a sequence of Lipshitz $\sigma^{(n)}$ converging unifomly 
to $\sigma$ 
and consider 
the sequence of solutions $u^{(n)}, v^{(n)}$ corresponding to $\sigma^{(n)}$. It is a standard to
 check that the sequence is tight and any weak limit point satisfies our equations. Since comparison is 
 satisfied 
for each $n$ it also holds in the limit. 

2.  There exists a probability space with a noise $\dot W$ and a pair of independent noises $\dot W_1$ and $\dot W_2$ such that $u$ solves~(\ref{rkpp}), $v$ solves (\ref{drkpp2}) and $\tilde w$ solves 
\begin{eqnarray}\label{3.8aa}
\partial_t \tilde{w}=\partial^2_x\tilde{w} + f(\tilde{w}+v)-f(v) + \tilde\sigma_w \dot{W}_2 + \delta_{x-\kappa t}\dot{A},
\end{eqnarray}  
where $\tilde\sigma^2_w=|\sigma^2(v+\tilde{w})-\sigma^{2}(v)|$, $\tilde w$ is non-negative and 
\begin{eqnarray}
\label{uvw_order}
 u\leq v+\tilde w.
\end{eqnarray}
  The construction of such a tripple  $(u, v, \tilde w) $ is fairly
straightforward.  
One constructs a sequence of approximations to~(\ref{rkpp}), (\ref{drkpp2}) and (\ref{3.8aa}) for which the 
ordering correspondent to~(\ref{uvw_order}) is elementary. Then one takes a limit to get solutions to 
  (\ref{rkpp}), (\ref{drkpp2}) and (\ref{3.8aa}) and shows that ordering is preserved in the limit. By this 
way one gets that the the unique weak solution $u$ to~(\ref{rkpp}) is bounded from the above by $v+\tilde w$
 where $v, \tilde w$ are {\it some}
 solutions to (\ref{drkpp2}) and (\ref{3.8aa}) respectively with independent 
 white noises $\dot W_1, \dot W_2$.

Define
\begin{eqnarray}
\tilde u &=& v+\tilde w,\;\;t\leq \tau_1\\
\partial_t \tilde{u} 
&=& \partial_x^2 \tilde{u} + f(\tilde{u}) + \epsilon\sigma(\tilde u)\dot{W},\;t\geq \tau_1,
\end{eqnarray} 
where $\tau_1$ is defined similarly to $\tau$:
\[ \tau_1= \inf\{ t\ge 0 : \exists x,~ \tilde\sigma_w^2(t,x)=0, ~ (v+\tilde w)(t,x)\neq v(t,x)\}.\]
It is easy to see that $\tilde u$ is a solution to the equation which $\bar u$ is supposed to solve, 
and hence we can set $\bar u = \tilde u$ and $\tau_1=\tau$.
To show that $\bar u$ and $v$ indeed belong to ${\cal C}_{\rm exp}$ one can use for example the methods of proof of Theorem~1.2 from~\cite{mps06}. 
Now let us construct $\bar w$ satisfying~(\ref{e30}) such that 
\[ \tilde w\leq \bar w,\;\;{\rm on}\; t\leq \tau.\]
Let $w$ be a solution to 
\begin{eqnarray}
\partial_t w=\partial^2_x w + (w +\tilde w) - (f(\tilde{w}+v)-f(v)),\;\;t\leq \tau.
\end{eqnarray} 
As the drift term is non-negative we get that $w$ is non-negative. Now define $\bar w=w +\tilde w$
 and it is 
easy to check that it satisfies~(\ref{e30}) and we are done. \end{proof}

   \section{Large deviations}
\setcounter{equation}{0}
\label{ldevs}

We now present a fairly standard type of large deviation result which covers the estimates
we need both in the upper and lower bounds.  We need some notation.  Let  $g(s,y,t,x)$ and $\eta(x,y)$ be deterministic, and
\begin{equation}\label{4p1}
\Gamma_{b,T}=\{ (t,x)  :  t\in [0,T], x-{\rm v} t\in [b-1,b] \}.
\end{equation}
For $(t,x)$ and $(t',x') $ in $\Gamma_{b,T}$ let
\begin{equation}
d((t,x),(t',x')) = |x'-x| + |t'-t|^{1/2}.
\end{equation}
Define
\begin{equation}\label{43a}
{\mathcal B}(g,\eta,b) =\sup_{(t,x),(t',x')\in\Gamma_{b,T}\atop
d((t,x),(t',x') ) \le 1}\frac{\int\int_0^\infty  [g(s,y,t', x')-g(s,y,t,x)]^2 \eta(s,y) dsdy}{d((t,x),(t',x'))}.
\end{equation}

\begin{lemma}\label{ldlemma1} Let  $g(s,y,t,x)$, $\eta(s,y)$, $\Gamma_{b,T}$ and ${\mathcal B}(g,\eta,b)$ be as above and $\sigma(t,x)$  nonanticipating with  
\begin{equation}
|\sigma(t,x)|^2\le
\eta(t,x), \qquad (t,x)\in [0,T]\times \bf R
\end{equation} almost surely, and define
\begin{equation}\label{zeddd}Z(t,x)=\int\int_0^t  g(s,y,t, x) \sigma(s,y) W(dsdy).
\end{equation}
There exist $C_{(\ref{lambdacond})},C_{(\ref{conc})}<\infty$ such  that if $T\ge 1$   and
 \begin{equation}\label{lambdacond}
\rambda \ge  C_{(\ref{lambdacond})}T{\mathcal B}^{1/2}
 \end{equation}
 Then, with $\Phi(d) = \sqrt{d(1+|\log_2 d|) }$,
\begin{equation}\label{conc}
P\left( \sup_{(t,x), (t',x')\in \Gamma_{b,T}} \frac{| Z(t',x')-Z(t,x)|}
{\Phi(d((t,x),(t',x'))) }  \ge \rambda\right) 
\le 4T\exp \left\{ -  C^{-1}_{(\ref{conc})}\rambda^2 T^{-2}{\mathcal B}^{-1}\right\}.
\end{equation}\end{lemma}

\begin{proof}  
Let $\mathcal{G}_n$ be the vertices  of an affine lattice with edges $\mathcal{E}_n$ parallel to the boundaries of $\Gamma_{b,T}$
and with edge lengths $2^{-n}$ in the $(1,0)$ direction and with 
vertical component $2^{-2n}$ in the $(1,{\rm v})$ direction. Let $\mathcal{G}=\cup_{n=0}^\infty\mathcal{G}_n$ and 
$\mathcal{E}=\cup_{n=0}^\infty\mathcal{E}_n$.  

Given 
$(t,x), (t', x')\in\mathcal{G}$ with \begin{equation}
\min(d((t,x),(t',x')),1) \in (2^{-n_0}, 2^{-(n_0-1)}]
\end{equation} there exists a path between them using edges from $\mathcal E$,
  which uses only edges from $\mathcal{E}_n$ with $n\ge n_0$,
and uses at most $T$ edges from any given $\mathcal{E}_n$.

For $e=(p,q)\in\mathcal{E}_n$,  write $Z_e:= Z(p)-Z (q)$ and $d_e=d(p,q)$. 
By standard It\^o calculus, 
\begin{equation}
E[ \exp\{ \gamma Z_{e}\}] \le \exp\{\frac12 \gamma^2  d_e{\mathcal B}\}.
\end{equation} Let $a_n=(10\sqrt{2}T)^{-1} (n+1)^{1/2}2^{-n/2}$ 
and \begin{equation}
\mathcal{A}_e = \{Z_e \leq a_n\rambda\}.
\end{equation}
By Chebyshev's inequality
\begin{equation}
P(\mathcal{A}_e^c)
\le \exp\{ \frac12\gamma^2  d_e {\mathcal B}- 
\gamma a_n\rambda \} .\end{equation}
Optimising the inequality over $\gamma$ gives\begin{equation}
\label{412}P(\mathcal{A}_e^c)
\le \exp\{- \frac12\rambda^2 a_n^2 d_e^{-1}{\mathcal B}^{-1} \} .\end{equation}
Let $
\mathcal{A}=\bigcap_{n=0}^\infty\bigcap_{e\in\mathcal{E}_n}\mathcal{A}_e$.
On $\mathcal{A}$,
\begin{eqnarray}
|Z(t',x')-Z(t,x)|  & \leq  &T
\rambda\sum_{n=n_0}^\infty a_n
\\
& \le & 
\rambda (n_0+1)^{1/2}2^{-n_0/2}\\
&\le & \rambda  \Phi(d((t,x),(t',x')))  .
\end{eqnarray}
Here we use the fact that for $n_0\ge 0$,
$\sum_{n=n_0+1}^\infty n^{1/2} 2^{-n/2}
 \le 10(n_0+1)^{1/2} 2^{-n_0/2}$.
Now we have
\begin{equation}
P(\mathcal{A}^c)\leq\sum_{n=0}^{\infty}\sum_{e\in\mathcal{E}_n}
P(\mathcal{A}_e^c)\le\sum_{n=0}^{\infty} |{\mathcal E}_n| P(\mathcal{A}_e^c).\end{equation}
 It is simple to check
that  $|{\mathcal E}_n| \le 2 T2^{3n}$ and $d_e\le  2^{-n+2}$.   From (\ref{412}) then,
\begin{equation}P(\mathcal{A}^c)\leq 2Te^{-2^{-10}\rambda^2T^{-2} {\mathcal B}^{-1}} (1- e^{3\log 2-2^{-10} \rambda^2T^{-2} {\mathcal B}^{-1}})\end{equation}
which gives (\ref{conc}) as long as (\ref{lambdacond}) holds.
Since $Z(s,y)$ is continuous, it is enough to check the bound on dyadics, and hence this completes the proof.
\end{proof}

In order to apply Lemma we need a bound on (\ref{43a}).
This is provided by the next lemma.  The lemma will only be applied with the $\lambda$ defined in
(\ref{214}), but it is true for other $\lambda$ satisfying (\ref{ldcond56}).

\begin{lemma}
\label{variance-lemma}
 Let \begin{equation}\label{418}g(s,y,t,x) = e^{a(t-s)} 1(0\le s\le t) G_{\rm v}(s,y,t,x)\end{equation}
 and \begin{equation}\eta(s,y)\le (1+ |y-{\rm v} s|) \exp\{ \lambda |y-{\rm v} s|\}.
 \end{equation}
 Then, there exists $C_{(\ref{lbA})}<\infty$ such that  for any $b<0$ (i) if $T\le 1$ and $a\le 1$, or (ii) if 
$T>1$ and 
\begin{equation}\label{ldcond56}
{\rm v}\lambda - \frac{\lambda^2}{2}\ge 2a , 
\end{equation}
 ${\mathcal B} $ from (\ref{43a}) satisfies  
\begin{equation} 
\label{lbA} {\mathcal B}\le C_{(\ref{lbA})} (1+|b|) e^{\lambda |b|} .
\end{equation}
The same also holds in case (i) if $G_{\rm v}$ is replaced by $G_{{\rm v}, L}$.
\end{lemma}

{\it Proof.} The only statement that is not elementary is (ii).
We have to estimate
\begin{eqnarray}\label{defofc} &&\int\int_0^\infty  [g(s,y,t+h, x+z)-g(s,y,t,x)]^2 (1+ |y-{\rm v} s|) \exp\{ \lambda |y-{\rm v} s|\} dsdy\nonumber\\ && \le C (1+|b|) e^{\lambda |b|}[z + h^{1/2}]
\end{eqnarray}
for $  t\in [0,T]$, $x-{\rm v} t\in [b-1,b]$, with $g(s,y,t,x)$ in (\ref{418}).
First of all, note that we can express $G_{\rm v}$ in terms of $G_0$, which in
turn can be written explicitly in terms of the heat kernel $G$ (see (\ref{e17})) using reflection;
\begin{equation}
\label{new-eq-6}
G_{\rm v}(s,y,t,x) = e^{-\frac{{\rm v}}{2}((x-{{\rm v}}t)-(y-{{\rm v}}s)) -\frac{{\rm v}^2}{4}(t-s)}
          G_{0}(s,y-{{\rm v}}s,t,x-{{\rm v}}t),
\end{equation}
\begin{equation} \label{reflprin}
G_{0}(s,y,t,x)=G(s,y,t,x)-G(s,y,t,-x), \qquad x,y<0.
\end{equation}
After  change of variables, the left hand side of (\ref{defofc}) becomes, with $x'=x-{\rm v} t$
and $\g = z-{\rm v} h$,
\begin{eqnarray}&&
\int_{y\le 0}\int \Big( {\mathbf 1}_{0\le s\le t+h} e^{ - \frac{{\rm v}}2 
  ( x' -y  + \g) - \a (t-s+h)} G_0(s,y, t+h, x'+\g) \nonumber \\ 
&& \qquad  -  {\mathbf 1}_{0\le s\le t} e^{ - \frac{{\rm v}}2 ( x' -y  ) - \a
  (t-s)} G_0(s,y, t, x') \Big)^2 (1+|y|) e^{\lambda |y|} dsdy.\label{firstest}
\end{eqnarray}
with $\a=\frac{{\rm v}^2}4-a$.
Estimating the two pieces of the right hand side of  (\ref{reflprin})
by using that the square of the sum is bounded by twice the sum of the squares 
we see that (\ref{firstest}) is bounded by the sum over $\iota=\pm1$ of
\begin{eqnarray}&&
 \int_{y\le 0} \int \Big( {\mathbf 1}_{-h\le s\le t}  
   \frac{\exp\{ - \frac{\rm v}2 ( x' -y  + \g) - \a
(s+h)- \frac{ (x'- \iota y+\g)^2}{4(s+h)}\}}{\sqrt{s+h}} \nonumber \\ 
&& ~~  -  {\mathbf 1}_{0\le s\le t}\frac{ \exp\{ - \frac{{\rm v}}2 ( x' -y  ) 
   - \a s- \frac{ (x'- \iota y)^2}{4s}\}}{\sqrt{s}}  \Big)^2(1+|y|) 
   e^{-\lambda y} \frac{dsdy}{4\pi}.
\label{firstest2}
\end{eqnarray} 
Note that we have also changed variables $t-s\mapsto s$.  Changing
$y\mapsto y+ x'$ and rearranging a little this becomes  
$(A_1+A_0 |x'|)e^{-\lambda x'}$ where $A_i$ is $(4\pi)^{-1}$ times the sum 
over $\iota=\pm1$ of
\begin{equation}
 \int\!\!\!\int 
e^{-2\a s+({\rm v}-\lambda) y}(1+\iota |y|)
\Big({\mathbf 1}_{-h\le s\le t}
\frac{ 
\g' e^{ 
-\frac{ (y+ (1-\iota) x - \iota \g)^2}{ 4(s+h)} } } {\sqrt{s+h} }
 - {\mathbf 1}_{0\le s\le t}  \frac{ e^{-\frac{ (y+ (1-\iota) x )^2}{ 4s} } } {\sqrt{s} }
  \Big)^2  dsdy\label{secest}
\end{equation}  
with $\g'=  \exp\{ 
\frac{{\rm v}\g}{2} -\a h\}$. Consider the $\iota=+1$ term.  We estimate  
  \begin{eqnarray} &&
\Big({\mathbf 1}_{-h\le s\le t}
\frac{ 
\g' e^{ 
-\frac{ (y -  \g)^2}{ 4(s+h)} } } {\sqrt{s+h} }
 - {\mathbf 1}_{0\le s\le t}  \frac{ e^{-\frac{ y^2}{ 4s} } } {\sqrt{s} }
  \Big)^2 \le 3\cdot{\mathbf 1}_{0\le s\le t}\Big(
\frac{ 
 e^{ 
-\frac{ (y -  \g)^2}{ 4(s+h)} } } {\sqrt{s+h} }
 -   \frac{ e^{-\frac{ y^2}{ 4s} } } {\sqrt{s} }
  \Big)^2\nonumber
\\
 &&+
  3\cdot{\mathbf 1}_{-h\le s\le 0}
\frac{ 
\g'^2 e^{ 
-\frac{ (y -  \g)^2}{ 2(s+h)} } } {s+h }+3\cdot{\mathbf 1}_{0\le s\le t} 
\frac{ 
(\g' -1)^2e^{ 
-\frac{ (y -  \g)^2}{ 2(s+h)} } } {s+h }.\label{frank}
 \end{eqnarray}
 There are also three analogous terms corresponding to $\iota=-1$.
 All six terms are estimated by explicit computation.  Since it is
 very tedious, we present only the worst case which is the first
 term on the right hand side of (\ref{frank}) with $\iota=1$.  Call $\b= {\rm v}-\lambda$.

\begin{lemma}
For $h,\g \in (0,1)$,  $\a,\b\in (0,3)$ and \begin{equation}\cu=2\alpha-\beta^2>0,  \end{equation}
 there exists a $C_{(\ref{carlsest})}<\infty$ such that for all $t>0$,
 \begin{equation}\label{carlsest}
 \int_0^t\int e^{-2\a s+\b y}(1+ |y|) 
\left(\frac{  e^{ 
-\frac{ (y -  \g)^2}{ 4(s+h)} } } {\sqrt{s+h} }
 -   \frac{ e^{-\frac{ y^2}{ 4s} } } {\sqrt{s} }
  \right)^2 dy ds \le C_{(\ref{carlsest})}[ |\gamma| + |h|^{1/2} ] 
  \end{equation}
\end{lemma}
\begin{proof}
The left hand side is bounded by a constant multiple of 
$\i_1+\i_2 + \i_3 (\g ^{2}-h)+\i_3 (h)  + \i_4
$
 where
\begin{eqnarray*}\i _{1} &=&
\int_{h}^{t}\int e^{-2\a s+\b y}(1+|y|)
    \left(\frac{e^{-\frac{y^2}{4(s+h)}}}{\sqrt{s}}
         - \frac{e^{-\frac{y^2}{4s}}}{\sqrt{s}}\right)^2 dyds\\
         \i _{2} &=&\int_{h}^{t}\int e^{-2\a s+\b y}(1+|y|)
    \left(\frac{e^{-\frac{y^2}{4(s+h)}}}{\sqrt{s+h}}
         - \frac{e^{-\frac{y^2}{4(s+h)}}}{\sqrt{s}}\right)^2 dyds \\
\i_3 (a) &=& \int_{0}^{|a|}\int
    e^{-2\a s+\b y}(1+|y|)\left(
    \frac{e^{-\frac{y^2}{2(s+h)}}}{s+h}  +     \frac{e^{-\frac{(y-\g )^2}{2s}}}{s} \right)dyds \\
\i _{4} &=&  \int_{0\vee(\g ^{2}-h)}^{t}\int
     e^{-2\a s+\b y}\frac{(1+|y|)}{s+h}
    \left(e^{-\frac{(y-\g )^2}{4(s+h)}}
        - e^{-\frac{y^2}{4(s+h)}}\right)^2 dyds  \\
\end{eqnarray*}
In the proof $C$ will denote any finite constant, possibly depending on $\a,\b$ and $\cu$.  Its value
will change from line to line.

{\it Estimation of $\i_1\le Ch^{1/2}$.}  By the mean value theorem, there exists $\theta\in[0,1]$ such that
\begin{equation}
\i _{1} = \frac{1}{16} h^2\int_{h}^{t}\int e^{-2\a s+\b y}(1+|y|)
 s^{-1}e^{-\frac{y^2}{2(s+\theta h)}} y^4(s+\theta h)^{-4}dyds.
\end{equation}
Since we are integrating over $s\in[h,t]$, we have 
$s\leq s+\theta h\leq 2s$ and therefore
\begin{eqnarray*}
\i _{1} &\leq
& Ch^2\int_{h}^{t}\int e^{-2\a s+\b y}
  e^{-\frac{y^2}{4s}} (y^4+|y|^5)s^{-5}dyds\\ &=&  Ch^2\int_{h}^{t}
    e^{-{\cu }s}s^{-5} \int 
  e^{-\frac{1}{4s}(y-2s\b )^2} (y^4+|y|^5) dyds.
\end{eqnarray*}Recall that
$
|a+b|^n\leq 2^{n-1}(|a|^n+|b|^n)$.  
Thus 
$y^4  = ((2s)^{\frac{1}{2}}z + 2s\b )^4 \leq 2^{5}z^4s^2 + 2^{7}s^4\b ^4 $,
$|y|^5  = |(2s)^{\frac{1}{2}}z + 2s\b |^5 \leq 2^{\frac{13}{2}}z^5s^{\frac{5}{2}} 
           + 2^{9}s^5\b ^5
$.
  After the change of 
variables change variables to 
$
z = (y-2s\b)/\sqrt{2s}$ and integrating we can bound the last term by
\begin{equation}
Ch^2 \int_{h}^{t}
    e^{-{\cu }s}       \left(s^{-\frac{5}{2}}+s^{-2}+
      s^{-\frac{1}{2}}\b ^4+s^{\frac{1}{2}}\b ^5\right) ds 
\end{equation}
where $C$ is a universal constant.
Since $s>h$ in the region of integration and $h\leq1$, this is bounded above by
\begin{equation}
Ch^2\int_{h}^{t}
    e^{-{\cu }s}      \left(h^{-\frac{1}{2}}s^{-2}+s^{-2}+
      h^{-\frac{1}{2}}\b ^4+h^{\frac{1}{2}}\b ^5\right) ds 
\end{equation}
The estimate then follows from 
$
\int_{h}^{t} e^{-{\cu }s}s^{-2}ds 
\leq \int_{h}^{1} s^{-2}ds + \int_{1}^{t} e^{-{\cu }s}ds 
\leq \frac{1}{h} + \cu^{-1} $.

{\it Estimation of $
\i _{2} \leq  C h^{\frac{1}{2}}
$.}  By the mean value theorem, there exists $\theta\in[0,1]$ such that 
\begin{eqnarray*}
\i _{2} &=&\frac14 h^2\int_{h}^{t}\int e^{-2\a s+\b y}(1+|y|)
     e^{-\frac{y^2}{2(s+h)}}(s+\theta h)^{-3}dyds \\
&\leq&  Ch^2\int_{h}^{t}\int e^{-2\a s+\b y}(1+|y|)
     e^{-\frac{y^2}{2(2s)}}s^{-3}dyds\\
&=& Ch^2\int_{h}^{t}
    e^{-{\cu }s}s^{-\frac{5}{2}}
    (\pi^{\frac{1}{2}}+2s^{\frac{1}{2}}+2\pi^{\frac{1}{2}}s\b)ds
\end{eqnarray*}
The estimate then follows from 
$
\int_{h}^{t} e^{-{\cu }s}s^{-\frac{5}{2}}
    (1+s^{\frac{1}{2}}+s)ds
\leq \int_{h}^{\infty} 
    \left(s^{-\frac{5}{2}} +  s^{-2}
    +s^{-\frac{3}{2}}\right) ds    
\leq C h^{-\frac{3}{2}}
$.

{\it Estimation of  $
\i_3 (\gamma^2-h)+ \i_3 (h)\leq
C[h^{\frac{1}{2}}+\gamma]
$.}   This just uses 
$
\int(1+|y|)e^{-\frac{(y-a)^2}{2r}}dy 
\leq \int(1+a+|y-a|)e^{-\frac{(y-a)^2}{2r}}dy
\le C((1+a)r^{\frac{1}{2}}  +r)
$
and $4\alpha-\beta^2>0$.

{\it Estimation of $\i _{4} \leq C\g $.}  First we change variables to $r=s+h$, with $h$ constant, and use the fact that 
$2\a r\geq0$ and $t<\infty$ to see that  
\begin{equation}
\i _{4} \leq e^{2\a h}\int_{h\vee\g ^{2}}^{\infty}\int
	I(r,y) dydr 
\label{est-A2-1}
\end{equation}
where
\begin{eqnarray}
\label{est-A2-2}
I(r,y) &=&  	r^{-1}  (1+|y|)
    \left(e^{-\frac{(y-\g-br)^2}{4r}}e^{\frac{\g \b}{2}}e^{\frac{\b^2r}{4}}
        - e^{-\frac{(y-\b r)^2}{4r}}e^{\frac{\b^2r}{4}}\right)^2     \nonumber\\
\label{est-A2-4}
&\leq& 2e^{-(2\a  + \frac{\b^2}{2})r}r^{-1}(1+|y|)
    e^{-\frac{(y-\g-\b r)^2}{2r}}\left(e^{\frac{\g \b}{2}} - 1\right)^2  \\
&& + 2e^{-(2\a  + \frac{\b^2}{2})r}r^{-1}(1+|y|)
    \left(e^{-\frac{(y-\g-\b r)^2}{4r}}
        - e^{-\frac{(y-\b r)^2}{4r}}\right)^2   \nonumber\\
&=:&	 I_1(r,y) + I_2(r,y)  \nonumber
\end{eqnarray}
We have
$
(e^{\frac{\g \b}{2}} - 1)^2  \leq C\g^2
$ and furthermore $\int (1+|y|)
    e^{-\frac{(y-\g-\b r)^2}{2r}}dy \le  \int (1+\g+\b r + |y-\g-\b r|)  e^{-\frac{(y-\g-\b r)^2}{2r}} dy
    \le C [r^{1/2}  (1+\g+\b r) + r]$ 
so, 
\begin{equation}\int_0^\infty \int I_1(r,y) dydr  \le C \gamma^2.
\end{equation}
Similarly we bound $\int I_2(r,y) dy
$ by 
\begin{equation}\label{fro}
  2\int
   e^{-(2\a  + \frac{\b^2}{2})r}
 r^{-1}(1+\b r+\frac{\g}{2}+|y-\b r-\frac{\g}{2}|)
    (e^{-\frac{(y-\g-\b r)^2}{4r}}
        - e^{-\frac{(y-\b r)^2}{4r}})^2 dy.  
\end{equation}
We need two standard estimates: For $r>0$,
\begin{eqnarray}
\label{gd:1} &&
\int (e^{-\frac{(y-\g)^2}{4r}}-e^{-\frac{y^2}{4r}})^2dy
\leq C r^{-\frac{1}{2}}\g^2\\ && \label{gd:2}
\int |y-\frac{\g}{2}|
    (e^{-\frac{(y-\g)^2}{4r}}-e^{-\frac{y^2}{4r}})^2dy
\leq C(r^{\frac{1}{2}}\g + \g^2)
\end{eqnarray}
Changing variables in (\ref{fro}) to $z=y-\beta r$ and using 
(\ref{gd:1}), (\ref{gd:2}) we have
\begin{equation}
\label{final-eq-lemma-8}
\int_{\g ^{2}}^{\infty}\int I_2(r,y) dydr  
\leq   C\g 
\end{equation} which gives the required estimate for $\i_4$  and completes the proof of the lemma.  
\end{proof}

This gives us a lemma which controls the large deviations on a long
time interval.

\begin{lemma} \label{lemma9a}Let \begin{equation}\label{zedupper}
Z (t,x) =  \int \int_0^t  e^{-(t-s)} G_{\rm v} (s,y,t,x) \sigma(s,y)  W(dsdy)
\end{equation} where 
$\sigma$ is nonanticipating.  
Assume  \begin{equation}
\sigma^2(t,x)\le \bar F(x-{\rm v} t) + 3e^{-{\lambda}(x-{\rm v} t)}.
\end{equation}
Then for $T$ be as in (\ref{teeee}),
\begin{equation}
\label{vFunder3}
P(\exists (t,x)\in [0,T]\times\mathbf R, x\le {\rm v} t :   Z (t,x)\ge \epsilon^{-1}e^{-{\lambda} (x-{\rm v} t) })
\le 1/16
\end{equation} 
and with $M$ as in (\ref{eeeem}),
\begin{equation}\label{oldset} P\big(\sup_{0\le t\le T, ~{\rm v} t-M \le x\le {\rm v} t
} |Z(t,x)| \ge \epsilon^{-1}u^*/160\big) \le 1/32
\end{equation}
\end{lemma}

\begin{proof}
The left 
    hand side of  (\ref{vFunder3}) is bounded by 
    \begin{equation}\label{239}
\sum_{n=0}^\infty P(\exists (t,x)\in \Gamma_{-n,T} :  Z (t,x)\ge \epsilon^{-1} e^{{\lambda}n })
\end{equation}
where $\Gamma_{-n,T}$ is defined in (\ref{4p1}).
Applying Lemma \ref{ldlemma1} for each
    $n$ with 
\begin{equation}\rambda =\epsilon^{-1}e^{{\lambda} n }/2\sqrt{T\log T}\end{equation} using 
 $t'=0$ and $\Phi \le 2\sqrt{T\log T}$  and (\ref{lbA}),
%
%
we obtain that 
 (\ref{239}) is bounded by
    \begin{equation}\label{239a}
4|\log\epsilon|^4\sum_{n=0}^\infty  \exp\{ -A(n+1)^{-1}
e^{{\lambda} n}\} \end{equation} 
with 
$ A=\epsilon^{-2}C_{(\ref{239a})}^{-1} |\log\epsilon|^{-12}|\log\log\epsilon^{-1}|^{-1} $
as long as 
\begin{equation}\rambda =\epsilon^{-1}e^{{\lambda}n }/2\sqrt{T\log T}\ge C_{(\ref{lambdacond})} T{\mathcal B}^{1/2}\end{equation}
From (\ref{214}), ${\lambda}-1\ge 1$ if $\delta>0.2$.  Hence
$(n+1)^{-1}e^{{\lambda} n}
\ge n+1$ for all $n\ge 0$ and (\ref{239a})
is bounded above by
 \begin{equation}\label{239b}
4|\log\epsilon|^4\sum_{n=0}^\infty  \exp\{ -A
(n+1)\} \le 8|\log\epsilon|^4e^{-A}\end{equation} 
as long as $A\ge \log 2$
which  bounded by $1/16$ for $\epsilon\le \epsilon_0$. 
\end{proof}

\section{Proof of Lemma \ref{lem1}}
\setcounter{equation}{0}
\label{lemma t-new}    
    

Recall that  $v$ and $\varrho$ are the solution of (\ref{drkpp}) and 
(\ref{compeq}) with initial data $\bar F$ as in (\ref{F}).
$\gamma, L, T, M$ are as in (\ref{epar2}), (\ref{elldef}), (\ref{tog}).


Lemma \ref{lem1} is basically a result about how the stochastic perturbation of a partial differential 
equation (\ref{drkpp})  stays close to its deterministic version.   Such theorems
are fairly standard, but we need to stay close on a fairly long time interval $[0,T]$
where $T={\mathcal O}(|\log\epsilon|^4)$ as required by (\ref{236}).
First of all, let $\bar v$ be the solution of \begin{equation}
\label{drkpp3}\left\{\begin{array}{ll}
\partial_t \bar v 
= \partial_x^2 \bar v +  \bar f(\bar v) + \epsilon\sigma(\bar v)\dot{W_1}, & x<{\rm v} t, \\
\bar v(t,x) = 0,& x\ge {\rm v} t, \end{array}\right.
\end{equation}  with $\bar v(0,x) = \bar F(x)$ with $\bar v \ge v$
by the comparison theorem, and let $\bar\varrho$ be the solution of
(\ref{compeq3}).  It suffices to show that
\begin{equation}\label{prechi2}
P\left(\exists t\in [0,T] :   \bar v(t,x)>\bar \varrho(t, x)+   3 {\lambda}
e^{-{\lambda}(x-{\rm v} t)}{\rm ~for~some~}x\in\mathbb R
\right)\le 1/16.
\end{equation}  We have
\begin{equation}\label{234a}
\partial_t (\bar{v}-\w) = \partial_x^2 (\bar{v}-\w) + \bar f(\bar{v})-\bar f(\w) + \epsilon \sigma(\bar{v}) \dot W_1\qquad x<{\rm v} t
\end{equation}  and $\bar{v}-\w=0$ on
$x\ge {\rm v} t$ and $t=0$.  One  easily checks that 
\begin{equation}
\label{eq:1}
\bar f(\bar{v})-\bar f(\w) \le 2- (\bar{v}-\w).
\end{equation}
Using the same ideas as in the proof of Proposition~\ref{comparison} we will show now that 
\begin{equation}
\label{eq:3}
\bar v-\w \le y
\end{equation}
with $y$ a solution of
\begin{equation}
\label{eq:2}
\partial_t y = \partial_x^2 y + 2- y + \epsilon \sigma(v)\dot W_1.
\end{equation} on $x\le {\rm v} t$, and $y=0$ otherwise. To prove this define $\tilde{w}$ to be the solution to 
$$
\partial_t \tilde{w} = \partial_x^2 \tilde{w} + 2- (\bar{v}-\w) - (\bar f(\bar{v})-\bar f(\w))
 -\tilde{w}.
$$
on  $x\le {\rm v} t$, and $\tilde w=0$ otherwise. Note that by~(\ref{eq:1}), $2- (\bar{v}-\w) - (\bar f(\bar{v})-\bar f(\w))\geq 0$ and hence $\tilde{w}\geq 0$. Now define 
$$ y= \bar{v}-\w +\tilde w,$$
and by trivial calculations we get that $y$ satisfies~(\ref{eq:2}) and since $\tilde w\geq 0$, (\ref{eq:3}) 
follows. 

Using the integrating factor $e^{t}$ we obtain
\begin{equation}
\label{236second}\bar{v}(t,x)-\w(t,x) \le 
  \epsilon  Z(t,x)  +2
\end{equation}
where, with $G_{\rm v}(s,y,t,x)$ as in (\ref{e26}),
\begin{equation}
Z (t,x) =  \int \int_0^t  e^{-(t-s)} G_{\rm v} (s,y,t,x) \sigma(s,y)  W_1(dsdy)
\end{equation}
 So it suffices to show that
\begin{equation}
\label{vFunder4}
P(\exists (t,x)\in [0,T]\times\mathbf R, x\le {\rm v} t :   Z (t,x)\ge \epsilon^{-1}e^{-{\lambda} (x-{\rm v} t) })
\le 1/16.
\end{equation} 
Note that when we do this we can assume without loss of generality that 
\begin{equation}\label{bdonsig}
\sigma^2(t,x)\le \bar\varrho(t,x) + 3e^{-{\lambda}(x-{\rm v} t)}.
\end{equation}
For if $\tilde v$ is a solution of (\ref{drkpp2}) with $\sigma^2(\tilde v(t,x))$ replaced by $$
\tilde\sigma^2(\tilde{v},t,x)=\min\left(\sigma^2(\tilde{v}(t,x)),\varrho(t,x)+3e^{-{\lambda}(x-{\rm v} t)}\right)$$ then $\bar v=\tilde v$ up to time 
\[\hat \tau=\inf\{t\ge0~:~ \bar v(t, x) \ge \varrho(t,x)+3e^{-{\lambda}(x-{\rm v} t)} {\rm ~ for ~some~ }  ~x\in\mathbf R\}\,\]
and hence it suffices to prove (\ref{vFunder4})  under (\ref{bdonsig}).  
The result now follows from Lemma \ref{lemma9a}. \qed

 \section{Proof of Lemma \ref{lemma2}}\setcounter{equation}{0}
\label{lemma t-new-second}   

    $i.$ of Lemma \ref{lemma4} is actually an upgrade of Lemma \ref{lem1} which is a bit stronger than Lemma \ref{lemma2}.  $ii.$ is similar to $i.$ It is needed in Section \ref{killing} to control the maximum of $w$.   Recall $\bar F$ from (\ref{F}) and define \begin{equation}\label{underlineF}
\underline{F}(x) = \left\{ \begin{array}{l} 0,\;\; x\ge 0,\\ 
                     u^*/40,\;\; x\in(-M,0],\\ 
                      \bar F(x) + 3{\lambda}e^{-{\lambda} x}, \;\;x\leq -M, \end{array} 
                      \right.
                      \end{equation}  
                      
                      \begin{lemma}\label{lemma4} $i.$
                     Let $v$ be the solution of (\ref{drkpp}) with initial
data (\ref{F}).  Let $\varrho$ be the solution of (\ref{compeq3}) with the
same initial data.  Let $\gamma,L,T,M$ be as in (\ref{eeeem})-(\ref{teeee}).  Then
\begin{equation}\label{prechi2second}
P\left(\exists t\in [0,T] :   v(t,x)>\underline F(x-{\rm v} t){\rm ~for~some~}x\in{\bf R}
\right)\le 1/8.
\end{equation}
$ii.$  
\label{lem:1}
Suppose that  $u$ satisfies $0\le  u(0,x)\le \underline{F}(x)$ from (\ref{underlineF}) and
\begin{eqnarray}\label{up}
\partial_t  u\le  \partial_x^2  u +  u+ 
\epsilon \sigma(t,x) \dot{W}
\end{eqnarray}with 
\begin{equation}
\label{620a}
\sigma(t,x)^2\leq 3 \bar F(x-{\rm v} t -L).\end{equation} 
Suppose  $M$, $\gamma$ are as in (\ref{eeeem}), (\ref{epar2}).  Then there is a $C_{(\ref{620first})}<\infty$ such that  \begin{eqnarray}
\label{620first}
P\big(\sup_{0\leq t \leq 3, ~ -1\le x-{\rm v} t \le 1} u(t,x) >
u^*/10\big)\leq C_{(\ref{620first})}\gamma 
\end{eqnarray}
\end{lemma}
Let us make a few remarks before we start the proof of the lemma. Note that part (i) of the lemma is only stronger than Lemma \ref{lem1} in the region $x-{\rm v} t \in (-M, 0]$ where $\bar\varrho(t,x)
 + 3{\lambda}e^{-{\lambda}(x-{\rm v} t)}\sim 3$. We need this to be able to control $\sigma^2(v)$ from below 
 in the regions where $v$ is small, that is in the region $x-{\rm v} t \in (-M, 0]$. If we will show that with 
 high probability $v$ is small in that region, then we will be able to use (\ref{sigmacond}) to control $\sigma^2(v)$ 
 from below there. Another remark deals with coefficient $3$ in~(\ref{620a}). This coefficient appears in~(\ref{816}) 
 after which we use Lemma~\ref{lemma4}.

 \begin{proof}      $i.$ 
Note also that by
 the same argument as that at (\ref{bdonsig}) we can assume that
 \begin{equation}\label{newbdons}
 \sigma^2(v(t,x)) \le \underline F(x-{\rm v} t).
 \end{equation}
 Let $\mathcal N$ be the vertices of an affine lattice in $$\Gamma= \{(t,x)~:~ 
0\le t\le T, ~{\rm v} t-M\le x\le {\rm v} t\}$$ with edge length $a\varepsilon$ between nearest neighbour vertices.  From (\ref{vlew}) we have
that for $(x,t)\in \Gamma$ and therefore for $p\in\mathcal{N}$,
\begin{equation}
E\left[{v}(p)\right] \leq \varrho \le \nu Me^M.
\end{equation}
By Markov's inequality,
\begin{equation}
\label{Markov-est}
P({v}(p)> u^*/40) \leq 
{40Me^M\nu}/{u^*}.
\end{equation}
We can estimate \begin{eqnarray}
 P\left(\sup_{p\in\mathcal{N}}{v}(p)>u^*/40\right) 
   \le \sum_{p\in\mathcal{N}}P\left({v}(p)>u^*/40\right) 
\end{eqnarray}
and since $|\Gamma|\le 2 MT\varepsilon^{-2}a^{-2}$, if 
$
a\ge 80 M^{1/2} e^{M/2} (u^*)^{-1/2}$
 then \begin{equation}
\label{Markov-est-a}
 P\left(\sup_{p\in\mathcal{N}}{v}(p)>u^*/40\right) 
   \le 1/32.
\end{equation}
Hence to prove the lemma it suffices to show that if
$
a \le C_{(\ref{64})}^{-1}\epsilon^{-1}u^*
$ then
\begin{equation} P\left(\sup_{(t,x), (t',x')\in \Gamma
\atop |x'-x|+ |t'-t|\le a\varepsilon} |v(t',x')-v(t,x)| \ge u^*/40\right) \le 1/16.
\end{equation}
Divide  $\Gamma$ into $T$ intervals of length $1$, $\Gamma^i=\Gamma\cap \{ i\le t\le i+1\}$.
On $\{ i\le t\le i+1\}$,  ${v}$ is the solution to\begin{equation}
{v}(t,x)   =  \int G_{ \rm v}(i,y, t,x) {v}(i,y)dy + 
\int \int_{i}^t G_{ \rm v}(s,y, t,x)  f( {v}(s,y)) ds dy + \epsilon Z(t,x).
\end{equation} where
\begin{equation}
Z(t,x)= 
\int \int_{i}^t G_{\rm v}(s,y, t,x) \sigma( {v},s,y) W_1(ds dy).
\end{equation}
From (\ref{newbdons}), assuming $t'\ge t$,
\begin{equation} 
|{v}(t',x')-{v}(t,x)| \le \Omega_1+\Omega_2 + \Omega_3 + \epsilon | Z(t',x')|+ \epsilon |Z(t,x)|.
\end{equation} where \begin{eqnarray}
\Omega_1 & = &  3\int |G_{ \rm v}(i,y, t',x')-G_{ \rm v}(i,y, t,x)|
 e^{-{\lambda}( y-{\rm v} i)} dy,\\
 \Omega_2 & = &  3\int\int_i^{t_1} |G_{ \rm v}(s,y, t',x')-G_{ \rm v}(s,y, t,x)|
 e^{-{\lambda}( y-{\rm v} s)} dsdy,\\ 
 \Omega_3 & = &  3\int\int_{t}^{t'} G_{ \rm v}(s,y, t',x')
 e^{-{\lambda}( y-{\rm v} s)} dsdy. \end{eqnarray}
   So the result follows from (\ref{oldset}) and the elementary fact that
 there exists $c<\infty$ such that if  $|t-t'| + |x-x'|\le cu^*$,  $\lambda,{\rm v}\in [1,2]$ then 
 \begin{equation}\label{64}
 \Omega_1+\Omega_2 + \Omega_3 \le u^*/80.
 \end{equation}

$ii.$  From (\ref{up}), \begin{equation}u(t,x) \le e^{t}\int   G(0,y,t,x) \underline{F}(y) dy + \epsilon e^t Z(t,x)
\end{equation} 
where
\begin{equation} Z(t,x) = \int\int_0^t e^{-s} G(s,y,t,x) 
\sigma(s,y) W(dsdy)\end{equation} 
From the definition of $\underline{F}$ it is clear that we can choose $M$ so that for all $0\le t\le 3$ and $-1\le x-{\rm v} t \le 1$, 
\begin{equation}
\int  e^t G(0,y,t,x) \underline{F}(y) dy < u^*/20.
\end{equation} so the result follows from the following   large deviation estimate
whose proof is elementary as it only has to hold on time intervals of order $1$:
There exists a $C_{(\ref{620first})}<\infty $ such that for $\gamma$ is as in (\ref{epar2}),
\begin{eqnarray}
P\big(\sup_{0\leq t \leq 3, ~ -1\le x-{\rm v} t \le 1} \epsilon e^t |Z(t,x)| >
u^*/20\big)\leq C_{(\ref{620first})}\gamma.
\end{eqnarray}
\end{proof}

%

\section{The critical mass}
\label{outline}
\setcounter{equation}{0}



%

The  following
elementary computation identifies the critical mass for survival.
\begin{lemma}
\label{lem:l4}
Let $\dot W$ be a white noise and  
$w(t,x)$ be a positive solution of 
\begin{equation}\label{frst}
\partial_t w = \partial_x^2 w + b w+ \vartheta\sqrt{w}\dot W, \quad w(0,x)=w_0
\end{equation}
where $\vartheta$ is adapted with $\vartheta\ge \vartheta_0$ for some nonrandom
$\vartheta_0>0$.
Then 
\begin{eqnarray}
P(w(t)\equiv 0)\geq 1- e^{bt}\vartheta_0^{-2} t^{-1} E[\int w_0(x) dx]
\end{eqnarray}
\end{lemma}

\begin{proof} By considering $\tilde w=e^{-bt} w$ we can assume without loss of generality that $b=0$.  
If 
\begin{equation}\label{scnd}
\partial_t\phi =\partial_x^2  \phi - \vartheta_0^2 \phi^2
\end{equation}
then
\begin{eqnarray}&
\exp\{ -\int \phi(t-s,x), w(s,x)dx\rangle\}
&\end{eqnarray}
is a supermartingale in the $s$ variable on $[0,t]$. 
The solution of (\ref{scnd}) with $\phi(0,x)=n$ is 
$
\phi(t,x) = (\vartheta_0^2 t + n^{-1} )^{-1}
$
 and hence
\begin{eqnarray}&
E[ \exp\{ -n\int w(t,x)dx\}~|~\mathcal F_0] = \exp\{ - (\vartheta_0^2 t + n^{-1} )^{-1}\int w(0,x)dx\}.&
\end{eqnarray}
Taking $n\rightarrow\infty$  we get 
\begin{eqnarray}
\label{equt:4} &
P(\int  w(t,x)dx =0) 
 =  E[ \exp\{-\vartheta_0^{-2}t^{-1}\int w(0,x)dx\}]
& \end{eqnarray}
and the lemma follows from $e^{-x} \ge 1-x$.  \end{proof}

The next  lemma is needed to control the support of such
a $w$ in short time intervals, in terms of the immigration.  Note that
we will only have to have reasonable control, and the actual scale
are not critical here, as it is in the previous lemma.  

\begin{lemma}
\label{lemmaB1}
Let $W$ be a white noise and $w$ be a solution of  
\begin{equation}\label{spud}
\partial_t w = \partial_x^2 w + b w + \vartheta\sqrt{w} \dot W+ d\mu, \qquad
0\le t\le 1
\end{equation}
with $w(0,x)\equiv 0$ and let $\psi$ be the hitting time of $(-r,r)^c$;
\begin{equation}
\psi = \inf\{ t\ge 0~:~ {\rm supp}( w(t)) \cap (-r,r)^c\neq 0\}.
\end{equation}  Suppose that $\mu$ is a positive adapted measure on $[0,1]\times \mathbf R$ with 
support in  $[0,1]\times (-r/2,r/2) $, and
 $\vartheta$ is adapted and $\vartheta\ge \vartheta_0>0$. 
Then, letting $M=\int\int_0^1 \mu(dt dx)$,
\begin{eqnarray}&
P(\psi >1) \ge  1- 100r^{-2}e^{b}\vartheta_0^{-2}E[ M].
&\end{eqnarray}
\end{lemma}

\begin{proof}  By considering $\tilde w = \vartheta_0^2e^{-bt} w$ we can assume without loss of generality that 
$b=0$ and $\vartheta_0=1$.  Furthermore, by symmetry it is enough to prove the result when $\psi$ is  
the hitting time of $(-\infty, -r]$ with a constant $50$ instead of $100$ on the right hand side.   First let us consider the case $\mu(t)\in \mathcal F_0$. Note that for any $\delta>0$
\begin{equation}
g_\delta(x) =   12(x+r+\delta)^{-2} 
\end{equation}
satisfies 
$
\partial_x^2g_\delta = g^2_\delta/2
$
for $x>-r-\delta/2$. Let  \begin{eqnarray}&
X_\delta(t) = \exp\left\{-\int g_\delta(x) w(t,x)dx+ 
 \int\int_0^t  g_\delta(x) \mu(dsdx)\right\}.&
\end{eqnarray}
Then $X_\delta(t\wedge \psi)$ 
is a submartingale.  In particular,
\begin{equation}
E[ X_\delta(1\wedge\psi)~|~\mathcal F_0] \ge E[X_\delta(0)~|~\mathcal F_0] =1.
\end{equation}
Let us assume temporarily that $\mu(t)\in \mathcal F_0$, $0\le t\le 1$.
Since $g\ge 0$, and $g\le 48 r^{-2}$ on ${\rm supp}(\mu)$ we have
$M_{g_\delta}(1\wedge\psi)\le M_{g_\delta}(1)\le 48r^{-2}M$,
\begin{eqnarray}
\label{639} 
&
E\left[ \exp\left\{-\int g_\delta(x) w(1\wedge\psi,x)dx\right\}~|~\mathcal F_0\right] 
 \ge \exp\left\{ - 48r^{-2}M\right\}. &
\end{eqnarray}
Note that
\begin{eqnarray}
\label{640} 
& P(\psi>1) \ge \liminf_{\delta\to 0} E\left[ \exp\left\{-
\int g_\delta(x) w(1\wedge\psi,x)dx\right\}~|~\mathcal F_0\right]
&
\end{eqnarray}
which proves the lemma when $\mu(t)\in \mathcal F_0$.

For the general case, note first that (\ref{spud}) has the property that if
$w_1$ and $w_2$ are two solutions with  measures $\mu_1$ and $\mu_2$
and independent white noises $W_1$ and $W_2$,
then $w_1+w_2$ is a solution with measure $\mu_1+\mu_2$. 

We construct a probability space on which we have this setup with 
adapted $\vartheta$ and $\mu_i(dx,dt) = \mu_i(dx)\delta_{t_i}(dt)$. 
Let $\psi_1$ and $\psi_2$ be the corresponding hitting times of $(-\infty,-r]$.
From (\ref{639}) and (\ref{640}), conditioning on $\mathcal F_{t_i}$ instead of 
$\mathcal F_0$ we have $P(\psi_i>T) \ge \exp\{ -48r^{-2} M_i\}$ for $i=1,2$ 
with $M_i=\int \mu_i(dx)$ and hence 
$P(\psi_1\wedge \psi_2>1) \ge E[\exp\{ -48r^{-2} (M_1+M_2)\}]$.  

A finite induction then gives the result
for $\mu(dtdx) =\sum_{n=0}^N \mu_n(dx) \delta_{t_n} (dt)$
with $\mu_n(dx)\in\mathcal F_{t_n}$.  We can then take limits
to obtain the result for all adapted positive measures $\mu$.
\end{proof}
%


 \section{Proof of Lemma \ref{lemma3a}}
\label{killing}
\setcounter{equation}{0}
\newcommand{\gdm}{\hfill\vrule  height5pt width5pt \vspace{.1in}}
\newcommand{\bk}{{\rm v}}  
\newcommand{\tv}{\tilde{v}}

   We will solve (\ref{e30}) iteratively, on short time intervals of length
$1$ and show that we can kill the mass of $w$ on each interval separately.
The reason to do this is that the noise in (\ref{e30}), which is needed to kill the mass,
is only of the correct order near $x={\rm v} t$ where $v$ is relatively
small.  So one has to show that the mass vanishes quickly, before the front moves ahead, and the
noise is no longer available.  We will do all our bounds on the event 
$\{v(t, x) \le \underline{F}(t-{\rm v} t)\}$ where $\underline F$ is defined in (\ref{underlineF}). To be more precise, define
 \begin{equation}\label{tvdef}
\tau_v=\inf\{t\ge0~:~ v(t, x) \ge \underline{F}(x-\bk t) {\rm ~ for ~some~ }  ~x\in\mathbf R\}.
 \end{equation}
Let
\[ \tv(t,\cdot)= v(t\wedge \tau_v\,, \cdot).\]

Let $W_{2,k}\,, \;k=1,2,\ldots$ be a sequence of independent 
white noises which are also independent of $W_1$. We construct a sequence of processes $\bar{w}_k$,  $k=1,2,\ldots$ by solving\begin{equation}\left\{\begin{array}{ll} \partial_t \bar{w}_{k}=\partial_x^2 \bar{w}_k + \|f\|_{\rm Lip}\bar{w}_k + \epsilon\sigma_k   \dot{W}_{2,k}  + \delta_{x-\bk t} \dot A_k  & k-1 < t\le k+1,\\
\bar{w}_k(t,\cdot)=0                            & t= k-1,
                                     \end{array}\right.
  \end{equation}     
  and setting $ \bar{w}_k(t,\cdot)=0$ for $t\in (k-1,k+1]^c$.                             
Here
\begin{equation}\sigma_k^2 =
 |\sigma^2(\tv+\bar{w}_k+\bar{w}_{k-1})-\sigma^2 (\tv+\bar{w}_{k-1})|\vee a^*  \bar{w}_k
\end{equation}
and
\begin{equation}
\dot A_k = \dot A 1_{k-1<t\le k}
\end{equation}
is the creation term acting only on the first half of each time interval.
To start things going we use the convention that $\bar{w}_{-1}=\bar{w}_0\equiv 0$.
Define stopping times
\begin{eqnarray}
\mathcal \tau_{k,1} & = &\inf\{t\in (k-1,k+1]:  {\rm supp}\{\bar{w}_k(t)\}\not\subset (\bk (k-1) -1,\bk k+1)\}, \nonumber
\\\nonumber 
\mathcal \tau_{k,2} & = &\inf \{ t\in (k-2,k+1]: \bar{w}_k(t,x)+\bar{w}_{k-1}(t,x)+\tv(t,x)> u^*/10\\\nonumber && \qquad\qquad\qquad\qquad \qquad {\rm for ~some ~~} x\in (\bk (k-2) -1,\bk k+1)\} ,
\\\nonumber \mathcal \tau_{k,3} & = &\left\{\begin{array}{lll} \infty, &&{\rm if}\  \bar{w}_k (k+1 , \cdot) \equiv 0,\\
                                     k+1,&&{\rm otherwise}.
                                     \end{array}\right.
 \end{eqnarray}
with the convention that the infimum is infinite if the set is empty, and ${\rm supp}\{w\}=\{x: w(x)>0\}$ is the support of a non-negative function $w$.

Let $\tau= \tau_{k,i}$ be the smallest $\tau_{k,i}<T$, if there is one.  Otherwise let $\tau=T$. Note that up to time   $ \tau\wedge (k+1)$ 
we have
\[ \sigma_k^2 =\sigma^2(\tv+\bar{w}_k+\bar{w}_{k-1})-\sigma^2 (\tv+\bar{w}_{k-1}).\]
Hence, we can find a probability space on which there are white noises $W_2$ and 
$\{W_{2,i}\,, \;i=1,2\}$ such that the solution $\bar w$ of (\ref{e30a}) can be
represented as
\begin{equation}
\label{equt:3}
 \bar{w}(t,x)1(t\le \tau\wedge \tau_v)= 
 \sum_{k=1}^{\infty} \bar{w}_{k}(t,x)
 1( k-1 < t\le (k+1) \wedge \tau\wedge \tau_v ). 
\end{equation}
For each $k=1,\ldots,T$ let
\begin{equation}
\dirset_k = \{ \tau_{k,1}>k+1\}\cap \{  \tau_{k,2}> k+1\}\cap  \{
\tau_{k,3}> k+1\}
\end{equation}
Recall $\gamma$ from (\ref{epar2}).  We claim that
\label{lem:2}  \begin{equation}\label{712}
P( \cap_{k=1}^{T} \dirset_k )\geq 
1-c_0\gamma T.
\end{equation}
This implies Lemma \ref{lemma3a}, for on  $\cap_{k=1}^{T}\dirset_k\cap \{ \tau_v>T\}$,
\begin{equation} \bar{w}(t,x)= 
 \sum_{k=1}^{\infty} \bar{w}_{k}(t,x)
 1( k-1 < t\le k+1 ), \qquad t\leq T.
 \end{equation}   Hence the left hand side of (\ref{254aa}) is bounded above by 
\[ 1- P\left(\cap_{k=1}^T\dirset_k \cap \{ \tau_v>T\}\right)\leq c_0\gamma T+1/8,\] 
since $P(\tau_v<T)<1/8$ by Lemma \ref{lemma4}, and we are done. 

The rest of the section will be devoted to verifying~(\ref{712}). 
\
Clearly 
it suffices to prove that for each $k=1,\ldots, T$ and $i=1,2,3$,
\begin{equation}\label{8.4}
P(\tau_{k,i}\leq k+1 )\leq c_0\gamma.
\end{equation}
By Lemma \ref{lem:l4}, applied to $\bar{w}_k$ on the interval $(k, k+1]$ where there is
no creation acting on $\bar{w}_k$,
\begin{equation}\label{810}
P(\tau_{k,3}\leq k+1)
 \le C_{(\ref{810})}\epsilon^{-2}  E[ A(k) - A(k-1)] \le
 2C_{(\ref{810})}\gamma,
\end{equation}
where$C_{(\ref{810})} = e^{\|f\|_{\rm Lip}}(a^*)^{-1}$.
In the last inequality we used~(\ref{bdnearkt}). 
By the same reasoning, but using Lemma \ref{lemmaB1}
instead of Lemma~\ref{lem:l4}, there is a $C_{(\ref{equt:1})}<\infty$ such that 
\begin{equation}
\label{equt:1}
P( \tau_{k,1}\le k+1) \le  C_{(\ref{equt:1})}\gamma.
\end{equation}
It remains to prove (\ref{8.4}) for  $\tau_{k,2}$.  The rest of the proof is
devoted to this.
Define  
\begin{eqnarray}
u_k(t)=\bar{w}_k(t,x)+\bar{w}_{k-1}(t,x)+\tv(t,x), \quad t\in [k-2,k+1].
\end{eqnarray}
Given that $\bar{w}_k(k-1,\cdot)=\bar{w}_{k-2}(k-1,\cdot)=0$ there is
a white noise $\dot W_{(k)}$ such that
\begin{equation}
\partial_t u_{k}=\partial_x^2 u_k + (\bar{w}_k+\bar{w}_{k-1})+ f(\tv)+ 
 \epsilon\sigma_{(k)}   \dot{W}_{(k)}\,,
\end{equation}
where $\sigma_{(k)}^2 = \sigma_{k-1}^2  + \sigma^2_k   +  
 \sigma(\tv)^2$.
Since $f(\tv)\leq \tv$,  
\begin{equation}
\partial_t u_{k}\le \partial_x^2  u_k +  u_k+ 
 \epsilon\sigma_{(k)}\dot{W}_{(k)}\,. 
\end{equation}
The inequality is meant as holding for the corresponding integral
equation.
Now we claim that on $\{ \tau_{k,1}>k+1\} \cap \{\tau_{k-1,1}>k\}$, we have, for $t\le \tau_{k,2}$
\begin{equation}\label{816}
 \sigma_{(k)}^2(t,x)
 \leq  (\|\sigma^2\|_{\rm Lip}+ 2) \bar F(x-{\rm v} t-L). 
\end{equation}
For we know that there $\sigma^2(\tilde v(t,x))\le \tilde v(t,x) \le {\underline F}(x-{\rm v} t) \le \bar F(x-{\rm v} t-L)$.  Also $\sigma^2_k(t,x)\le \|\sigma^2\|_{\rm Lip}\bar w_k(t,x)\le \|\sigma^2\|_{\rm Lip}\frac{u^*}{10}
1_{\{(k-1,k+1]\times({\rm v}(k-1)-1,{\rm v} k+1)\}} \le \|\sigma^2\|_{\rm Lip}\bar F(x-{\rm v} t-L)$.

Note that \begin{equation}
\label{8.2} 
P(\tau_{k,2}\le k+1, \tau_{k,1}> k+1, \tau_{k-1,1}> k) 
\le P\Big(\sup_{k-2\leq t \leq k+1\atop {\rm v}(k-2)-1\leq x\leq \bk k +1)} 
 u_k(t,x) 
> u^*/10\Big)\,. 
\end{equation}
We want to estimate this with Lemma \ref{lem:1},  but we need (\ref{620a}), 
which does not necessarily hold.  But by the argument around (\ref{bdonsig}), 
in proving (\ref{620a}), we can assume without loss of generality that 
(\ref{816}) holds.  
In fact this is the place where it is clear the appearance of coefficient $3$ in (\ref{620a}). 
Hence we can apply Lemma \ref{lem:1} to obtain
\begin{eqnarray}
\label{8.2second}
P(\tau_{k,2}\le k+1, \tau_{k,1}> k+1, \tau_{k-1,1}> k)\le C_{(\ref{620first})}\gamma.
\end{eqnarray}
Finally, \begin{eqnarray}
\nonumber
P(\tau_{k,2}\le k+1)&\leq&P(\tau_{k,2}\le k+1, \tau_{k,1}> k+1, \tau_{k-1,1}> k)+ P(\tau_{k-1,1}\leq k)\\
\nonumber
&&\mbox{}+P(\tau_{k,1}\leq  k+1)\\
\nonumber
&\leq& (C_{(\ref{620first})}+ 2C_{(\ref{equt:1})})\gamma, 
\end{eqnarray}
which completes the proof of Lemma \ref{lemma3a}.

\section{Proof of Lemma \ref{lbd}}
\setcounter{equation}{0}
\label{lemma lbd} 

We need a preliminary result  of how a stochastic perturbation of a 
partial differential equation stays close to its deterministic version.  Here the interval 
is of order $1$, but the estimate needs to be precise.

\begin{lemma}  \label{lemma12}Suppose that $u$ and $\v$  are solutions of 
\begin{equation}
\label{rkpplower2}
\partial_t u 
= \partial_x^2 u  +f(u ) + \sigma(u)\dot{W} 
\end{equation} and \begin{equation}
\partial_t {\v} = \partial_x^2 { \v} +f(\v)  \end{equation}   
on $ |x|<L +{\rm v} t$, $0<t\le T$ with $
u (t,x) =\v(t,x)= 0$ on $ |x|\ge L+{\rm v} t$,  and  $u(0,x)=\rho(0,x)$. Suppose that $f$ is Lipschitz with constant $K$.
Then
\begin{equation}
|{u}(t,x)-{\v}(t,x)|\le  |\tilde Z(t,x)| + |Z(t,x)|. 
\end{equation}
where
\begin{equation}\label{ztx}
\tilde Z(t,x) =  K \int\int_0^t e^{K(t-s)} G_{{\rm v},L}(s,y,t,x) |Z(s,y) |dsdy
,
\end{equation} 
\begin{equation}\label{vtx}
Z(t,x) = \int\int_0^t G_{{\rm v},L}(s,y,t,x) \sigma(u(s,y) ){W} (dsdy).
\end{equation}\end{lemma}

\begin{proof}
Let $
D = u-{\v}-Z
$.
Note that $D$ satisfies\begin{equation}
\label{D-eq}
\partial_t D
= \partial_x^2 D + f(D+{\v}+Z )-  { f}({\v} )  \end{equation} on $
           |x|<L +{\rm v} t$ with 
 $D (t,x) = 0 $ on $  |x|\ge L+{\rm v} t$ and $t=0$.
Now $ |{ f}(D+{\v}+v )-  { f}({\v} )|\le K |D|  + K|v|$.
Let $D_+$ and $D_-$ be the solutions of\begin{equation}
\label{rkpplower21}
\partial_t D_\pm
= \partial_x^2  D_\pm \pm K[| D_\pm|  + |Z|] , \qquad
           |x|<L +{\rm v} t
           \end{equation} with 
$ D_\pm (t,x) = 0$ for $  |x|\ge L+{\rm v} t$ and 
$ D_\pm (0,x)=0 $.
Note that $D_+$ is a supersolution and $D_-$ is a subsolution of (\ref{D-eq}), so that 
$
D_-\leq D \leq  D_+$.  Furthermore, $ D_-\le 0\le D_+$, so (\ref{rkpplower21} ) can be
solved explicitly in terms of $Z$.  We get
\begin{equation}
 D_\pm (t,x)  =  \pm K \int\int_0^t e^{K(t-s)} G_{{\rm v},L}(s,y,t,x) |Z(s,y)|dsdy.
\end{equation} as desired.
\end{proof}

Now we continue with the proof of Lemma \ref{lbd}.  
Define
\begin{equation}\hat\v(t,x)=\underline\v(t,x)+ r \epsilon\sqrt{{\goo \underline\v(t,x) }}
\end{equation}
\begin{equation} h(x)=(1+|x|)\exp\{ (1-\frac{\delta}2)|x|\}
\end{equation}


\begin{proof}[Proof of Lemma \ref{lbd}] 
 Let 
\begin{equation}
\label{curlyB}\mathcal A=\left\{ \omega:|\underline{u}(t,x) -  \underline\v(t,x)|\le \epsilon r\sqrt{{\goo \underline\v(t,x)}}
 ,~0\le t\le 1, x\in \mathbf R\right\}.
\end{equation}
In particular, on $\mathcal A$,  we have 
\begin{equation}
\underline{u}(t,x)\le\hat\v(t,x)
,\qquad 0\le t\le 1, ~x\in \mathbf R.
\end{equation} If we were to let 
$
\tilde\sigma(t,x) = \sqrt{ \min\{\sigma^2(t,x),  \hat\v(t,x)\}},
$ and
$\tilde u$ be the solution of 
(\ref{rkpplower}) with $\sigma$ replaced by $\tilde \sigma$, and   
$\tilde{\mathcal A}$ the analogue of ${\mathcal A}$ with $u$ replaced by
$\tilde u$, then $P(\mathcal A) = 
P(\tilde{\mathcal A})$.
Hence in estimating $P(\mathcal A)$ we can assume without loss of generality
that \begin{equation}
\label{assupt}
\sigma^2 (t,x) \le \hat\v(t,x).
\end{equation}
$\underline{f}$ is Lipschitz with constant $1$, so by Lemma \ref{lemma12}, 
\begin{equation}
|\underline{u}(t,x) -  \underline\v(t,x)|\le  \epsilon |\tilde Z(t,x)| + \epsilon |Z(t,x)|
,
\end{equation} 
where $\tilde Z(t,x)$ and $Z(t,x)$ are is as in (\ref{ztx}) and (\ref{vtx}) with $u$ replaced by $\underline{u}$ and $K=1$.  

Now note that there exists $ C_{(\ref{911})}<\infty$ such that for $0\le t\le 1$,
\begin{equation} \label{911} \hat\v(t,x)\le C_{(\ref{911})} \epsilon^2 \gamma h(x-{\rm v} t+L).
\end{equation}
By Lemmas \ref{ldlemma1} and \ref{variance-lemma}  with  $T=1$,  and $a=1$, and by (\ref{911}),  we have,
for some $C_{(\ref{918})}<\infty$, 
 \begin{equation}\label{918}
P( \sup_{0\le t\le 1 \atop x-{\rm v}t \in [b-1,b]}  |Z(t,x)| \ge  r \epsilon\gamma^{1/2}\sqrt{h(b)} ) \le 4 \exp\{ -C_{(\ref{918})}^{-1}r^2\} .
\end{equation}
Furthermore there exists $C_{(\ref{913})}<\infty$ such that for $0\le t\le 1$, \begin{equation}\label{913}
\epsilon^2\gamma  h(x-{\rm v} t+L)\le C_{(\ref{913})}{\goo \underline\v(t,x) }\qquad x\le {\rm v} t+L
\end{equation}
so summing $b$ from $0$ to $L$ and using (\ref{913}),
\begin{equation}\label{920a}
P(  \epsilon|Z(1,x)| \le  r\epsilon \sqrt{\goo \underline\v(1,x) }{\rm ~~for~all~~} x\in[{\rm v},{\rm v}+L) )\ge 1- 4 L\exp\{ -C_{(\ref{920a})}^{-1} r^2\} .
\end{equation}
Finally, it is not hard to check that there is a $C_{(\ref{912})}<\infty$ such that for $0\le t\le 1$,
\begin{equation} \label{912}
\int\int_0^t e^{t-s} G_{{\rm v},L}(s,y,t,x) \sqrt{h(y+L)} dsdy\le C_{(\ref{912})}  \sqrt{h(x-{\rm v} t+L)},
\end{equation} and therefore we also have for some $C_{(\ref{921})}<\infty$,
\begin{equation}\label{921}
P( \epsilon |\tilde Z(1,x)| \le   r\epsilon \sqrt{\goo \underline\v(t,x) }{\rm ~~for~all~~} x\in[{\rm v},{\rm v}+L)) \ge 1- 4 L\exp\{ -C_{(\ref{921})}^{-1}  r^2\} ,
\end{equation}
which completes the proof of Lemma \ref{lbd}.\end{proof}
%


\begin{thebibliography}{McK76}

\bibitem[BDL]{BDL} R. D. Benguria, M. C. Depassier, M. Loss, Validity of the Brunet-Derrida formula for the speed of pulled fronts with a cutoff, arXiv:0706.3671.


\bibitem[Bra78]{bra78}
M. Bramson.
\newblock Maximal displacement of branching {B}rownian motion.
\newblock {\em Comm. Pure Appl. Math.}, 31(5):531--581, 1978.

\bibitem[Bra83]{bra83}
M. Bramson.
\newblock Convergence of solutions of the {K}olmogorov equation to travelling
  waves.
\newblock {\em Mem. Amer. Math. Soc.}, 44(285):iv+190, 1983.

\bibitem[BraDurr88]{bradurr88}Bramson, Maury; Durrett, Rick A simple proof of the stability criterion of Gray and Griffeath. {\em Probab. Theory Related Fields} 80 (1988), no. 2, 293--298.


\bibitem[BD97]{bd97}
E. Brunet, B. Derrida.
\newblock Shift in the velocity of a front due to a cutoff.
\newblock {\em Phys. Rev. E (3)}, 56(3, part A):2597--2604, 1997.

\bibitem[BD01]{bd01}
E. Brunet, B. Derrida.
\newblock Effect of microscopic noise on front propagation.
\newblock {\em J. Statist. Phys.}, 103(1-2):269--282, 2001.

\bibitem[BDMM]{BDMM} E. Brunet, B. Derrida, A. H. Mueller, S. Munier, A phenomenological theory giving the full statistics of the position of fluctuating pulled fronts,cond-mat/0512021


\bibitem[CD04]{cd04}
J.~Conlon and C.~Doering.
\newblock On travelling waves for the stochastic
  Fisher-Kolmogorov-Petrovsky-Piscunov equation.
\newblock  {\em J. Stat. Phys.} 120 (2005), no. 3-4, 421--477.

\bibitem[DPK]{DPK}
F. Dumortier, N. Popovic, T. Kaper \newblock The critical wave speed for the Fisher-Kolmogorov-Petrowskii-Piscounov equation with cut-off 	
\newblock {\em Nonlinearity},  Volume 20, Number 4, April 2007

\bibitem[D]{D}  R. Durrett, Oriented percolation in two dimensions,  {\it Ann.
Prob.}, 12 (1984), 999-1040.

\bibitem[F]{F}
R. A. Fisher, The wave of advance of advantageous genes, {\it Ann. Eugen.}, 7:355�69 (1937).


 \bibitem[GP]{GP}
A. Galves, E. Presutti, Edge fluctuations for the one-dimensional supercritical contact process. {\em Ann. Probab.} 15 (1987), no. 3, 1131--1145. 


\bibitem[HT]{HT} P. Horridge, R. Tribe, , On stationary distributions for the KPP equation with branching noise.  {\em Ann. Inst. H. Poincar Probab. Statist.}  40  (2004),  no. 6, 759--770.


\bibitem[Isc88]{isc88}
I.~Iscoe.
\newblock On the supports of measure-valued critical branching {B}rownian
  motion.
\newblock {\em Probab. Theory Related Fields}, 16:200--221, 1988.




\bibitem[KPP]{KPP} A. Kolmogorov, I. Petrovsky, and N. Piscunov, E tude de l'Equation de la diffusion
avec croissance de la quantite de mati`ere et son application `a un probl`eme biologique,{\it 
Bjull. Moskov. Gos. Univ. Mat. i Meh.} 1 (1937), no. 6, 1Ð25; Zbl 18, 321.

\bibitem[KS98]{kns98}
Z.~Kessler, D. A.~Ner and L.~M. Sander.
\newblock Front propagation:precursors, cutoffs and structural stability.
\newblock {\em Phys. Rev. E}, 58:107--114, 1998.

\bibitem[Lig85]{lig85}
T.M. Liggett.
\newblock {\em Interacting Particle Systems}.
\newblock Springer-Verlag, Berlin, Heidelberg, New York, 1985.

\bibitem[McK75]{mck75}
H.~P. McKean.
\newblock Application of Brownian motion to the equation of
  Kolmogorov-Petrovskii-Piskunov.
\newblock {\em Comm. Pure Appl. Math.}, 28(3):323--331, 1975.

\bibitem[McK76]{mck76}
H.~P. McKean.
\newblock A correction to: ``Application of Brownian motion to the equation
  of Kolmogorov-Petrovskii -Piskonov'' (Comm. Pure Appl. Math.
  {\bf 28} (1975), no. 3, 323--331).
\newblock {\em Comm. Pure Appl. Math.}, 29(5):553--554, 1976.

\bibitem[MP]{MP} 
L.~Mytnik, E.~Perkins,  Pathwise uniqueness for stochastic heat equations with H\"older continuous coefficients: the white noise case.  Preprint.

\bibitem[MPS06]{mps06}
L.~Mytnik, E.~Perkins, A.~Sturm.
\newblock On pathwise uniqueness for stochastic heat equations with
              non-{L}ipschitz coefficients.
\newblock {\em Ann. Probab.}, 34:1910--1959, 2006.


\bibitem[MP92]{Mueller-Perkins}  C. Mueller, E.Perkins,  The compact support property for solutions to the heat equation with noise. {\em Probab. Theory Related Fields }93 (1992), no. 3, 325--358.

\bibitem[MS93]{ms93}
C.~Mueller and R.~Sowers.
\newblock Blow-up for the heat equation with a noise term.
\newblock {\em Probab. Theory Related Fields}, 97:287--320, 1993.

\bibitem[MS95]{ms95}
C.~Mueller and R.~Sowers.
\newblock Random traveling waves for the {K}{P}{P} equation with noise.
\newblock {\em J. Funct. Anal.}, 128:439--498, 1995.

\bibitem[MT95]{mt95}
C.~Mueller and R.~Tribe.
\newblock Stochastic p.d.e.'s arising from the long range contact and long
  range voter processes.
\newblock {\em Probab. Theory Related Fields}, 102(4):519--546, 1995.
 
\bibitem[P]{P} D. Panja, Phys. Rev. E 68, 065202(R) (2003); Phys. 
Rep. 393, 87 (2004). 


\bibitem[Pa]{Pa} E.~Pardoux, . Stochastic partial differential equations, a review.  Bull. Sci. Math.  117  (1993),  no. 1, 29--47.

\bibitem[PL99]{pl99}
L.~Pechenik and H.~Levine.
\newblock Interfacial velocity corrections due to multiplicative noise.
\newblock {\em Phys. Rev. E}, 59:3893--3900, 1999.

\bibitem[Shi88]{shi88}
T.~Shiga.
\newblock Stepping stone models in population genetics and population dynamics.
\newblock In S.~Albeverio et~al., editor, {\em Stochastic Processes in Physics
  and Engineering}, pages 345--355. D. Reidel, 1988.

\bibitem[Sow92]{sow92a}
R.~Sowers.
\newblock Large deviations for a reaction-diffusion equation with
  non-{G}aussian perturbations.
\newblock {\em Ann. Probab.}, 20:504--537, 1992.



\bibitem[vS]{vS} W.van Saarloos,   Front propagation into unstable states, {\em Physics Reports}, 386 29-222 (2003).
 
\bibitem[Wal86]{wal86}
J.B. Walsh.
\newblock An introduction to stochastic partial differential equations.
\newblock In P.~L. Hennequin, editor, {\em \'Ecole d'\'et\'e de probabilit\'es
  de Saint-Flour, XIV-1984}, Lecture Notes in Mathematics 1180, pages 265--439,
  Berlin, Heidelberg, New York, 1986. Springer-Verlag.

\end{thebibliography}

\end{document}